\documentclass[11pt,a4paper]{article}
\newcommand{\Format}{article}

\usepackage{etoolbox}
\newtoggle{format-template}
\newtoggle{format-article}
\togglefalse{format-template}
\togglefalse{format-article}
\toggletrue{format-\Format}
\usepackage[utf8]{inputenc}% Encodage des caractères accentués dans le code source
\usepackage[T1]{fontenc}%    Encodage des caractères accentués dans le fichier de sortie
\usepackage{lmodern}%        Police plus complète que celle de LaTeX, et proche de la standard
\usepackage{lastpage}%       Numérote les pages par rapport à la dernière (forme k/n)
\usepackage{xcolor}%         Possibilité de mettre du texte en couleur
\usepackage{ragged2e}%       Auto-justification du texte, même dans les environnements comme Beamer
\usepackage[square,numbers]{natbib}%Gestion de la bibliographie
\usepackage[english]{babel}% Charge la prise en compte de la typographie
\usepackage{lipsum}%         Importe \lipsum{k} pour mettre k gros paragraphes de texte vides de sens
\usepackage{xspace}%         Ajout de la commande \xspace, créant un espace insécable
\usepackage{fancyhdr}%       Contrôle efficace des en-tête
\usepackage{geometry}
\usepackage{enumitem}%       Gestion des itemize
\usepackage{graphicx}%       Gestion des commandes \includegraphics
\usepackage{subfigure}%      Gestion de sous-figures (avec leur propre légende, etc.) dans une figure
\usepackage{wrapfig}%        Placement de figures sur un bord d'une page, à côté d'un texte
\usepackage{float}%          Import de la commande [H] (float) dans l'emplacement des images
\usepackage{hyperref}%       Gestion des références internes dans le document
\usepackage{caption}%        Gestion des références entre éléments (figures, tableaux, algos, ...)
\usepackage{epstopdf}%       Permet d'utiliser des .eps avec pdfLaTeX
\usepackage{textcomp}%       Divers symboles génériques (\perthousand, etc.)
\usepackage{gensymb}%        Divers symboles génériques (\ohm, \degree, etc.)
\usepackage{eurosym}%        Ajout de la commande \euro{}
\usepackage{listings}%       Insertion de code dans le fichier de sortie, rendu propre au langage
\usepackage{algorithm}%      Commandes d'écriture d'algorithmes
\usepackage[noend]{algpseudocode}%Suppression des lignes "End" dans les algos en pseudo-code
\usepackage{amsmath}%        Commandes vitales pour écrire des maths (def des opérateurs, etc.)
\usepackage{dsfont}%         Écriture de divers caractères en version "maths" (double barre, comme IR)
\usepackage{amsthm}%         Définit un environnement "théorème" pour les maths
\usepackage{amssymb}%        Gère certains caractères mathématiques dans le fichier de sortie
\usepackage{stmaryrd}%       Ajoute quelques symboles utiles en logique mathématique
\usepackage{mathtools}%      Ajoute d'autres symboles mathématiques
\usepackage{tikz}%           Création d'images vectorielles
\usepackage{media9}%         Gestion de fichier audio/vidéo dans le document (Pour Beamer, surtout)
\usepackage[textwidth=2.2cm, backgroundcolor=yellow, linecolor=orange, textsize=scriptsize]{todonotes}
\algrenewcommand{\Return}{\State\algorithmicreturn~}% Simplifie l'écriture d'un Return en pseudo-code
\algnewcommand{\Not}[1]{\textbf{not}(#1)}% Commande pour le "Not" algorithmique en gras
\algnewcommand{\And}{\textbf{and}}%        Commande pour le "And" algorithmique en gras
\DeclareMathOperator*{\argmin}{arg\,min}%  Commande unifiée pour argmin
\newcommand{\dsum}[2]{\sum\limits_{#1}^{#2}}%       Écriture compacte d'une somme
\newcommand{\dprod}[2]{\prod\limits_{#1}^{#2}}%     Écriture compacte d'un produit
\newcommand{\pdiff}[2]{\dfrac{\partial{#1}}{\partial{#2}}}% Dérivée partielle
\newcommand{\abs}[1]{\left\lvert{#1}\right\rvert}% Valeur absolue
\DeclarePairedDelimiter\norm{\lVert}{\rVert}%      Redéfinition de la taille max d'une norme
\def\ll{\llbracket}%    Borne gauche de [[a,b]] l'intervalle des entiers entre a et b : [[
\def\rr{\rrbracket}%    Borne droite de [[a,b]] l'intervalle des entiers entre a et b : ]]
\def\o{\mathcal{o}}%    Opérateur "négligeable devant" (petit o)
\def\N{\mathbb{N}}%     Ensemble des entiers naturels IN
\def\Z{\mathbb{Z}}%     Ensemble des entiers relatifs IZ
\def\R{\mathbb{R}}%     Corps des nombres réels IR
\def\H{\mathbb{H}}%     Corps des quaternions IH
\def\D{\mathcal{D}}%    Domaine des solutions réalisables
\def\P{\mathcal{P}}%    Ensemble P
\def\Cache{\mathcal V}% Cache
\def\M{\mathcal M}%     Mesh
\def\Pr{\mathbb{P}}%    Probabilité
\def\E{\mathbb{E}}%     Espérance
\def\1{\mathds{1}}%     Indicatrice
\newcommand{\Normale}[2]{\mathcal N\left(#1,#2\right)}% Loi normale
\def\Rinf{\overline{\R}}% Ensemble IR barre (IR union +- infini)
\def\gps{\textsc{Gps}\xspace}%      GPS
\def\mads{\textsc{Mads}\xspace}%    MADS
\def\nomad{{\sf NOMAD}\xspace}%     NOMAD
\def\robustmads{\textsc{Robust-Mads}\xspace}%Robust-MADS
\def\ego{\textsc{Ego}\xspace}%      EGO
\def\dpmads{\textsc{DpMads}\xspace}%DPMADS
\def\mpmads{\textsc{MpMads}\xspace}%MPMADS
\def\dfo{\textsc{DFO}\xspace}%      DFO
\def\bbo{\textsc{BBO}\xspace}%      BBO
\newtheorem{definition}{Définition}[section]
\newtheorem{lemma}[definition]{Lemma}

\newtheorem{theorem}[definition]{Theorem}
\newtheorem{corollary}[definition]{Corollary}

\newcommand{\comment}[1]{}%         Pour faire ignorer des paragraphes par le compilateur
\newcommand{\guillemets}[1]{\textquotedblleft#1\textquotedblright}% Guillemets LaTeX
\algrenewcommand\textproc{\textit}% Noms de fonctions algorithmiques pas en majuscules
\newcommand{\KeyWords}{derivative-free, blackbox, stochastic, noisy, adaptive precision, tunable preci-
sion, direct-search, Monte-Carlo simulation}% Mots-clé associés au document
\newcommand{\Titre}{Optimization of noisy blackboxes with adaptive precision}%  Titre du document
\newcommand{\Sujet}{Optimization of noisy blackboxes with adaptive precision}%  Sujet du document
\newcommand{\Auteur}{Alarie, Stéphane, Audet, Charles, Bouchet, Pierre-Yves, Le Digabel, Sébastien}%Nom du (des) auteur(s)
\hypersetup{
  pdftitle={\Titre},%     Titre du pdf généré
  pdfsubject={\Sujet},%   Sujet du pdf
  pdfauthor={\Auteur},%   Nom de l'auteur officiel du document
  pdfkeywords={\KeyWords},%Mots-clé associés au document
  bookmarksnumbered,%     Arborescence du document disponible à la lecture
  pdfnewwindow = true,%   Ouvre une nouvelle fenêtre si un lien du pdf envoie vers un autre document
  pdfstartview={FitB},%   À l'ouverture du pdf, colle la taille du document à celle de la fenêtre
  hidelinks,%             À la lecture, les liens ne sont pas entourés d'un rectangle coloré
  linktoc=all,%            Relie l'ensemble de la TableOfContent aux pages appropriées
  colorlinks = true,%     Met les liens en couleur
  linkcolor = blue,%      Couleur des liens internes réglée sur bleu
  citecolor = blue,%      Couleur des références vers des articles réglée sur bleu
  filecolor = black,%     Couleur des références vers des fichiers réglée sur noir
  urlcolor = blue,%       Couleur des liens url réglée sur bleu
  breaklinks = true%      Un lien peut apparaître sur deux lignes
}
\iftoggle{format-template}{
\pagestyle{fancy}%                    Ajout d'un header paramétrisable
\lhead{{\bf~Pierre-Yves Bouchet}}%    Contenu de la partie gauche   du header
\chead{}%                             Contenu de la partie centrale du header
\rhead{\textsc{Template}}%            Contenu de la partie droite   du header
\lfoot{}%                             Contenu de la partie gauche   du footer
\cfoot{\thepage~/~\pageref{LastPage}}%Contenu de la partie centrale du footer
\rfoot{}%                             Contenu de la partie droite   du footer
\voffset        +00.00 cm% Hauteur de l'espace mort en haut des pages  = 1 inch + \voffset
\hoffset        +00.00 cm% Largeur de l'espace mort à gauche des pages = 1 inch + \hoffset
\topmargin      +00.00 cm% Espace entre le bas de l'espace mort et le haut du header
\oddsidemargin  +00.00 cm% Espace entre la droite de l'espace mort et la gauche du cadre de texte
\evensidemargin +00.00 cm% Idem oddsidemargin, pour les pages paires d'une impression recto-verso
\headheight     +15.00 pt% Hauteur du header
\headwidth      +16.59 cm% Largeur du header
\headsep        +22.00 pt% Hauteur de l'espace entre le header et le texte
\textwidth      +16.59 cm% Largeur du bloc de texte
\textheight     +22.00 cm% Hauteur du bloc de texte
\marginparwidth +00.00 cm% Largeur du bloc de notes en marge à droite
\marginparsep   +00.00 cm% Largeur de l'espace séparant le bloc de texte et celui de marge
\parindent      +00.00 cm% Largeur d'une tabulation de début de paragraphe
\parskip        +00.25 cm% Hauteur de l'interligne entre deux paragraphes
\setlist[itemize]{noitemsep, topsep=-5pt}
\setlength{\abovecaptionskip}{0pt plus 0pt minus 0pt}
\setlength{\belowcaptionskip}{0pt plus 0pt minus 0pt}
}

\iftoggle{format-article}{
\geometry{letterpaper, tmargin=3cm, bmargin=3cm, lmargin=3.2cm, rmargin=3.2cm}
\setlength{\tabcolsep}{2.5pt}
\setlist[itemize]{parsep=0.5pt}
}

\title{
    {Optimization of noisy blackboxes with adaptive precision}
    \thanks{
    This work is supported by the NSERC CRD~RDCPJ~490744-15 grant and by an InnovÉÉ grant, both in collaboration with Hydro-Québec and Rio~Tinto.}
}

\author{
    \href{mailto:alarie.stephane@ireq.ca}{Stéphane Alarie}\thanks{
    Expertise -- Science des données et calcul haute performance, Institut de recherche d’Hydro-Québec, Varennes, Canada.
    }
    \and \href{mailto:Charles.Audet@gerad.ca}{Charles Audet}\thanks{
        {GERAD}
          and Département de mathématiques et génie industriel,
          Polytechnique Montréal,
          Montréal, Canada.
          }
    \and \href{mailto:pierre-yves.bouchet@polymtl.ca}{Pierre-Yves Bouchet$^\ddagger$}
    \and \href{mailto:sebastien.le.digabel@gerad.ca}{Sébastien Le Digabel$^\ddagger$}
}
\begin{document}
\maketitle

%===================================%
\noindent\textbf{Abstract:}
%===================================%

In derivative-free and blackbox optimization, the objective function is often evaluated through the execution of a computer program seen as a blackbox. 
It can be noisy, in the sense that its outputs are contaminated by random errors. 
Sometimes, the source of these errors is identified and controllable, in the sense that it is possible to reduce the standard deviation of the stochastic noise it generates. 
A common strategy to deal with such a situation is to monotonically diminish this standard deviation, to asymptotically make it converge to zero and ensure convergence of algorithms because the noise is dismantled. 
This work presents \mpmads, an algorithm which follows this approach. However, in practice a reduction of the standard deviation increases the computation time, and makes the optimization process long. 
Therefore, a second algorithm called \dpmads is introduced to explore another strategy, which does not force the standard deviation to monotonically diminish. 
Although these strategies are proved to be theoretically equivalents, tests on analytical problems and an industrial blackbox are presented to illustrate practical differences.

%===================================%
\KeyWords.

%===================================%
%===================================%
\section{Introduction}
\label{sec:Intro}
%===================================%

This work consider the unconstrained problem
\begin{eqnarray} \label{eq-problem}
    \underset{x\in D_f \subseteq \R^n}{\min} & f(x)
\end{eqnarray}
under the framework of \textit{blackbox optimization} (\bbo) and \textit{derivative-free optimization} (\dfo). In those frameworks is required almost no structure on the problem and $f$. It is frequently considered that $f$ is continuous on its domain $D_f\subset\R^n$ unknown \textit{a priori}, but its derivatives may not exist (in \bbo) or be impossible to evaluate (in \dfo). 
Usually, a \textit{blackbox} is a complex computer program, computationally intensive to run. 
This high level of complexity makes the derivatives nonexistent or difficult to estimate. 
Therefore, algorithms to minimise a blackbox problem do not rely on any gradient-based processes and use only models of the objective function or comparison of previously computed values of $f$ to decide the quality of a given point.

However, in some contexts the values of $f$ cannot be computed exactly. In the simulations performed by a program, it is possible that some stochasticity appears. 
For example, if the blackbox encodes a Monte-Carlo estimation of the value of interest, two executions of the program with the same input parameter $x$ may return different values. 
Then, the true value of $f(x)$ is unknown, as any attempt to compute it returns a value affected by some error. 
This source of errors makes any deterministic algorithm prone to failure, because it is assumed in their design that the computation of $f(x)$ is possible and accurate. 
When the source of stochasticity is known and its implementation in the program is intentional, it is sometimes possible to control its \textit{magnitude} through the standard deviation of the random error on the returned value. 
For example, when the program performs a Monte-Carlo estimation of a value, improving the number of Monte-Carlo draws used in the estimation statistically improves the quality of the returned estimate and reduces its standard deviation. 
This work refers to this situation as an \textit{adaptive precision program}. 
In this document, an \textit{adaptive precision blackbox} denotes
    a deterministic function $f$ which cannot be computed, 
    but may be approximated by an adaptive precision program. 
The computation time of an adaptive precision program depends on the magnitude of errors it ensures: 
    the lower the standard deviation is, the higher the computation time is. 
In Monte-Carlo estimations, the time grows as the inverse of the square of the standard deviation: the total computation time for $N$ Monte-Carlo draws is roughly $t \propto N$ while the standard deviation is of the form $\sigma \propto 1/\sqrt{N}$. One may consider that as a trade-off: at any execution of the program, any standard deviation can be guaranteed, but the cost can be prohibitive. 
Three paradigms tackle this added layer of complexity. Specific algorithms exist on each, most of these being extensions of existing deterministic strategies (overviews are given in~\cite{AuHa2017,CoScVibook}). 

The first possibility is to not control the magnitude of noise during the optimization, 
    but rather to decide it before the optimization starts. 
Under this strategy, any algorithm which deals with uncertainty can be used. 
Notably, algorithms designed for situations where the noise is not adaptive. 
However, it should be noted that deterministic algorithms have no guarantee to work, 
    even with this strategy. 
One can consider the \robustmads algorithm~\cite{AudIhaLedTrib2016} 
    which modifies the \mads algorithm~\cite{AuDe2006} 
    to create a smooth representation of the function, 
    knowing noisy estimates. 
\mads is also adapted as \textsc{Stoch-Mads} in~\cite{G-2019-30}, 
    an algorithm using probabilistic estimates to ensure convergence of a noisy problem where the noise variance is nor known neither adaptive. 
Various techniques from surface response design can also be used to dynamically generate a sequence of functions approximating $f$: 
    for example, the \textsc{Phoenics} solver~\cite{PHOENICS} 
    which uses Bayesian kernel density estimations, 
    or kriging, studied by Sasena~\cite{Sasena02}. 
A line search algorithm is developed by Paquette and Scheinberg~\cite{PaSc18}.
Also, some algorithms exploit trust-region principles, 
    like \textsc{Astro-Df} in~\cite{ASTRODF}.

The second possibility is to lead the optimization on a high magnitude of noise, 
    and reduces it monotonically during the optimization process. 
The algorithm from Polak and Wetter~\cite{PoWe06a} adapts the mechanics from the \gps
    algorithm~\cite{Torc97a} in the case where errors have a controllable upper bound.
Chen and Kelley~\cite{ChKe2016} propose a way to reduce the magnitude which can be used in direct-search algorithms when the adaptive precision program performs a Monte-Carlo estimation. 
They extend this strategy with the addition of a smoothing effect in~\cite{ChKeFeZa18}. 
A trust-region algorithm handling constraints and using Gaussian models, \textsc{SNowPaC}, is proposed in~\cite{snowpac}.
Rivier and Congedo proposes in~\cite{Rivier2019} a multi-objective framework. In~\cite{DOGS2019} is proposed a Delaunay triangulation in $\R^n$, refined jointly with the grow in precision.
Heuristics are also used, for example the modification of the Nelder-Mead algorithm~\cite{NeMe65a} proposed by Chang in~\cite{Ch2012}.

The last possibility is to adapt more frequently the magnitude of the noise, 
    reducing it when necessary and augmenting it when possible. 
Picheny \textit{et al.}~\cite{Picheny2012} propose an adaptation of the \ego algorithm~\cite{JoScWe97a} which uses adaptive precision programs with errors given by centred normal laws with controllable magnitude. 
This strategy is also used in multi-fidelity optimization, for example by Frandi and Papini in~\cite{FrPa2015} which uses a direct-search algorithm to a multi-fidelity context. 
Multi-fidelity optimization is also named \textit{simulation optimization} in some works, as in the review~\cite{amaran2017simulation}.

The present work introduces two modifications of the \mads algorithm~\cite{AuDe2006} called \mpmads (monotonic precision) and \dpmads (dynamic precision) to handle an adaptive precision blackbox problem with the last two paradigms.
The paper is divided as follows.
Section~\ref{section:NotationsAlgos} introduces the notations and summarises the pertinent elements from the \mads algorithm to introduce \dpmads and \mpmads.
The section then presents the two algorithmic variants and concludes by discussing practical implementation issues.
A convergence analysis is provided in Section~\ref{section:Proof}.
The main result provides necessary conditions that ensure that the algorithm produces, with probability one, a point at which the Clarke generalised directional derivatives are non-negative.
Computational experiments are performed in Section~\ref{section:NumericalAnalysis}
    on two analytical problems as well as on
    a real industrial problem.
The results demonstrates that \dpmads can considerably reduce the overall computational effort.

%===================================%
%===================================%
\section{Two precision-adapting algorithms}
\label{section:NotationsAlgos}
%===================================%

This section introduces notations 
 and proposes the monotone and dynamic precision algorithms \dpmads and \mpmads.

%===================================%
\subsection{Notations and mathematical optimization problem}
\label{section:Notations}
%===================================%

This work aims to solve the unconstrained optimization Problem~\eqref{eq-problem}. The domain $D_f$ on which $f$ is defined is unknown \textit{a priori}. 
As capturing this domain is part of the problem, one may consider the extreme barrier formulation of the objective proposed in~\cite{AuDe2006}. 
Denoting $\Rinf=\R\cup\{\pm\infty\}$, Problem~\eqref{eq-problem} can be reformulated with an extended definition of $f$:
\begin{equation}\setlength{\abovedisplayskip}{3pt}\label{eq:OptimProblemDet}
    \underset{x\in\R^n}{\min}\ f(x),~\text{with}~f: \left\{\begin{array}{ccl}
        \R^n & \longrightarrow & \Rinf \\
        x    & \longmapsto     & \left\{\begin{array}{ll}
                f(x)    & \textit{if } x \in D_f\\
                +\infty & \textit{otherwise.}\end{array}\right.
    \end{array}{}\right.
\setlength{\belowdisplayskip}{3pt}\end{equation}

In the context of this work, exact values of $f$ cannot be obtained. Only approximations may be computed, because of a \textit{noise} which alters the value $f(x)$ during the numerical evaluation. 
It is assumed that a \textit{stochastic noise} $\eta$ is added to the numerical value: 
    one observe $f(x)+\eta$ instead of computing $f(x)$. 
More precisely, while $\eta_\sigma\sim\Normale{0}{\sigma^2}$ 
    follows a centred normal law with standard deviation $\sigma$ independent from the point $x$ and the objective $f(x)$, 
    the noise is represented as $f_\sigma(x)=f(x)+\eta_\sigma$, 
    a random variable following $\Normale{f(x)}{\sigma^2}$. 
Any attempt to compute $f(x)$ returns $f_\sigma(x;\omega)$, 
    an observation $\omega$ of $f_\sigma(x)$. 
Therefore, two consecutive attempts to evaluate $f(x)$ may return two different outputs
    $f_\sigma(x;\omega_1) \neq f_\sigma(x;\omega_2)$. 
In addition to the normal behaviour of the noise,   
    the standard deviation $\sigma$ is assumed to be \textit{adaptive}. 
This signifies that its value is controllable by the one who attempts to evaluate $f(x)$. 
The value $\sigma$ may be modified for any evaluation, even if $x$ remains unchanged. 
Modifying the value $\sigma$ alters the random variable $f_\sigma(x)$ used to obtain approximation $f_\sigma(x;\omega)$ of $f(x)$. 
As a lower standard deviation $\sigma$ gives fewer probability that the estimate $f_\sigma(x;\omega)=f(x)+\eta_\sigma(\omega)$ highly differs from $f(x)$, the present work uses \textit{higher precision} as a synonym for \textit{lower standard deviation}. 
Furthermore, at an infinite precision (equivalently, a null standard deviation) $f_\sigma(x)$ converge to a Dirac measure centred on value $f(x)$:
    $\Pr\left(\underset{\sigma\rightarrow0}{\lim}~f_\sigma(x)=f(x)\right)=1$.

The precision required at any estimation of any $f(x)$ is determined by a so-called
    \textit{precision index} denoted $r\in\R$. 
This index is used to set the value of the standard deviation through the mapping function $\rho$: $\sigma=\rho(r)$. 
This function is assumed to be positive, upper bounded by a finite $\sigma_{max}$, and decreasing. 
Under these hypotheses, the index $r$ can be interpreted as the precision, 
    and $\rho(r)$ its associated standard deviation. 
Via $\rho$, a high value of precision index $r$ corresponds to a low standard deviation $\sigma$.

The \dpmads and \mpmads algorithms presented in Section~\ref{section:Algos} follow an iterative mechanic. 
The optimization process exploits, at any iteration denoted $k$, the computations performed by earlier iterations. The \textit{historic} at a point $x \in \R^n$, or \textit{cache at $x$}, is defined as follows: 
    $\Cache^k(x) = \left\{(f_\sigma(x;\omega),\sigma) \mid f_\sigma(x;\omega)~\text{has been observed during an iteration}~i \leq k\right\}$. 
Each element of this set is a couple where the first element is estimated value $f_\sigma(x;\omega)$ of $f(x)$ obtained from a noisy observation $f_\sigma(x)$, while the standard deviation $\sigma$ of the noise is the second element of the couple. 
If $x$ has not been evaluated up to iteration $k$, then $\Cache^k(x)=\emptyset$ is void. 
The set $\Cache^k(x)$ can be interpreted as the full historic at a given point $x$ at a given iteration $k$. One can also define the \textit{cache} at iteration $k$: 
$\Cache^k = \left\{\left(x,\Cache^k(x)\right) \mid x\in\R^n\right\}$, which links a point $x$ with the historic at $x$ up to iteration $k$.

These notations can be abusively extended in the following way: 
    $\Cache^k(x)$ is a function which returns the cache at $x$. 
This function searches in $\Cache^k$ for the couple formerly denoted $(x,\Cache^k(x))$, 
    and returns the second element. 
Then, $\Cache^k(x)$ returns the empty set $\emptyset$ if $x$ has never been evaluated up to iteration $k$, and the full historic at $x$ otherwise.

Given the cache $\Cache^k(x)$ at iteration $k$ and point $x$, it is possible to construct an estimation of the objective function on $x$. 
This estimate is denoted $f^k(x)$ and its statistical standard deviation $\sigma^k(x)$. 
Various techniques exist to create those, such as the maximum likelihood used in this work.
The value $f^k(x)$ is the best estimation of $f(x)$ that can be proposed up to iteration $k$. It is defined as the most plausible value of the mean of all the normal laws $f_\sigma(x)$ for which an observation $f_\sigma(x;\omega)$ is contained in $\Cache^k(x)$.
The value $f^k(x)$ is given by the formula below, where $(\lambda,\sigma)$ are elements of $\Cache^k(x)$, of the form $(f_\sigma(x;\omega),\sigma)$. 
Then, the statistical standard deviation of $f^k(x)$ can be computed as $\sigma^k(x)$, and $f^k(x)$ statistically follows a law $\Normale{f(x)}{\sigma^k(x)^2}$. 
No predicted or estimated values are proposed at non-evaluated points (points $x$ such that $\Cache^k(x)=\emptyset$). 
The estimates are:
\begin{equation}\label{eq:MaxLikelihood}\setlength{\abovedisplayskip}{3pt}\setlength{\belowdisplayskip}{3pt}
\left\{\begin{array}{rcll}
    f^k(x) & = & \dfrac{\sum_{(\lambda,\sigma)\in\Cache^k(x)}\lambda/\sigma^2}{\sum_{(\lambda,\sigma)\in\Cache^k(x)}1/\sigma^2} & \textit{if}~\Cache^k(x)\neq\emptyset,~+\infty~\textit{otherwise,} \\[2.5ex]
    \sigma^k(x) & = & \sqrt{\dfrac{1}{\sum_{(\lambda,\sigma)\in\Cache^k(x)}1/\sigma^2}} & \textit{if}~\Cache^k(x)\neq\emptyset,~+\infty~\textit{otherwise.}
\end{array}\right.
\end{equation}
These estimates are used to define the \textit{incumbent} at iteration $k$ as $x_*^k \in \underset{x \in \R^n}{\argmin} f^k(x)$.

The original Problem~\eqref{eq:OptimProblemDet} can be reformulated with no use of the values of $f(x)$ (values of the true objective function, which cannot be computed):
\begin{equation}\label{eq:OptimProblem}
    \underset{x\in\R^n}{\min}~\underset{\#\Cache^k(x)\rightarrow+\infty}{\lim}~f^k(x).
\end{equation}
Optima are unchanged, because of the almost-sure equality provided by strong law of large numbers which ensures that for all points $x$ estimated infinitely often as $k\rightarrow+\infty$, it is almost sure that the maximum likelihood $f^k(x)$ converges to $f(x)$:
\begin{equation*}
\comment{\setlength{\abovedisplayskip}{3pt}\setlength{\belowdisplayskip}{3pt}}
    \Pr\left(\underset{k\rightarrow+\infty}{\lim}~f^k(x)=f(x) \mid \#\Cache^k(x)\underset{k\rightarrow+\infty}{\longrightarrow}+\infty\right)=1,\quad\forall x\in\R^n.
\end{equation*}

%===================================%
\subsection{Adaptation of {\sc Mads} for adaptive precision control}
\label{section:MADSMechanics}
%===================================%

This work proposes two algorithms exploiting the spatial exploration mechanics given by the \mads algorithm to solve the noisy Problem~\eqref{eq:OptimProblem}. Recall that \mads uses of discretisations of the space $\R^n$ named \textit{meshes of size $\delta_m$ centred on $x$}. Such a mesh is defined as the set $\M_{\delta_m}(x)=\{x+\delta_mz \mid z \in \Z^n\}$.
At iteration $k$, \mads defines the mesh of size $\delta_m^k$ centred on its current incumbent $x_*^{k-1}$: $\M^k=\M_{\delta_m^k}(x_*^{k-1})=\{x_*^{k-1}+\delta_m^kz \mid z\in\Z^n\}$. 
From $\M^k$ is extracted a \textit{poll} $\P^k$, a set of \textit{candidates} points to be evaluated. The candidates remain close to the incumbent, on a \textit{frame} of size $\delta_p^k: \norm{x_*^{k-1}-x}_\infty\leq\delta_p^k,~\forall x\in\P^k$. 
Values of $\delta_m^k$ are chosen so that $\forall k,~\delta_m^k\leq\delta_p^k$, and therefore the rule $\delta_m^k=\min(\delta_p^k,(\delta_p^k)^2)$ from~\cite{AuDe2006} is chosen. 
A common strategy to efficiently explore the neighbourhood of the incumbent is to create a positive basis of $\R^n$ (denoted $\D^k$) such that $x_*^{k-1}+d\in\M^k~\textit{and}~\norm{d}_\infty\leq\delta_p^k,~\forall d\in\D^k$, as proposed in the \textsc{OrthoMads} algorithm in~\cite{AbAuDeLe09}.

As the precision is adaptive in this work, a cornerstone of both algorithms is the way they modify that precision. The precision $r$ can grow arbitrarily high to ensure a standard deviation $\sigma$ as low as desired. 
However, it impacts the computational cost per evaluation. Therefore, There is a trade-off to exploit in the best way: how to choose $r$ at any iteration and any point, to ensure the convergence within a computational effort as low as possible.

The first algorithm is \mpmads (\textit{Monotonic Precision \mads}), 
    a generalisation of the work of Polak and Wetter~\cite{PoWe06a}. 
Its behaviour gives a monotonic control of the precision: 
    the precision increases during the optimization process, 
    as slow as possible to avoid over-consumption of computational budget,
    but fast enough to ensure that the noise never impacts the convergence. 
At the end of any iteration $k$, \mpmads checks the quality of the estimates. 
If they are sufficiently accurate, 
    the precision index $r^{k+1}=r^k$ is left unchanged. 
Otherwise, it is increased ($r^{k+1}>r^k$) so that the standard deviation 
    $\sigma^{k+1}=\rho(r^{k+1})$ is sufficiently low to avoid the algorithm being misled by the noise.

The second algorithm is \dpmads (\textit{Dynamic Precision \mads}), 
    with a different control of the precision. 
\guillemets{Dynamic} means that the precision is not forced to increase. In \dpmads, $r^{k+1} < r^k$ is possible. 
This deteriorates the quality of future estimates, 
    but \dpmads ensures that the uncertainty coming from this reduction is sufficiently low to avoid biased convergence. 
At any iteration $k$, \dpmads attempts to set the precision $r^k$ at the lowest value possible which ensures that the standard deviation $\sigma^k=\rho(r^k)$ is sufficiently low to prevent biased decisions. 
In the algorithms, the $UpdateR(r^k,p^k)$ function modifies $r^k$, given $p^k$ an indicator of estimates quality. Detailed expressions of $UpdateR$ are given in Section~\ref{section:Implementation}.

The standard deviation $\sigma^k$ is used during iteration $k$ in the following way. 
\dpmads and \mpmads start the iteration $k$ from their current incumbent solution (a point which have the lowest estimate:
    $x_*^{k-1} \in \argmin\{f^{k-1}(x) \mid x\in\R^n\}$),
    and generate a \textit{poll} set $\P^k$ of candidates around $x_*^{k-1}$  
    (following the mechanics of \mads).  
Incumbent, as well as all the candidates, are evaluated so that 
    $\sigma^k(x)\leq\sigma^k,~\forall x\in\P^k\cup\{x_*^{k-1}\}$. 

The poll step on the set $\P^k$ is implemented in \dpmads and \mpmads via Algorithm~\ref{algo:poll}:

\begin{algorithm}
\caption{Poll step algorithm for non-deterministic objective}\label{algo:poll}
\begin{algorithmic}[1]
\Function{Poll}{$x_*^{k-1}$, $\delta_p^k$, $r^k$, $\Cache^{k-1}$}
\State $\sigma^k \gets \rho(r^k)$
\State $\delta_m^k \gets \min\left(\delta_p^k,\left(\delta_p^k\right)^2\right)$
\State $\M^k \gets \left\{x_*^{k-1}+\delta_m^k z \mid z\in\Z^n\right\}$
\State $\D^k \gets \textit{positive basis of $\R^n$ such that}~x_*^{k-1}+d\in \M^k~\textit{and}~\norm{d}_\infty\leq\delta_p^k,~\forall d\in\D^k$
\State $\mathcal{P}^k \gets \left\{x_*^{k-1}+d \mid d\in \D^k\right\}$
\For{$x\in \mathcal{P}^k\cup\left\{x_*^{k-1}\right\} \textit{such that } \sigma^{k-1}(x)>\sigma^k$}
    \State Compute some $(\sigma_1,\sigma_2,\dots,\sigma_u)$ so that $\sigma^k(x)=\left(\frac{1}{\sigma^{k-1}(x)^2}+\sum_{i=1}^u\frac{1}{\sigma_i^2}\right)^{-1/2} \leq \sigma^k$
    \State Generate $\left(f_{\sigma_i}(x)\right)_{1 \leq i \leq u}$, observe $\left(f_{\sigma_i}(x;\omega_i)\right)_{1 \leq i \leq u}$, update $\Cache^{k-1}(x)$ as $\Cache_{updated}^{k-1}(x)$
\EndFor
\State $\Cache^k \gets \left\{(x,\Cache_{updated}^{k-1}(x)) \mid x\in\R^n\right\}$
\State Choose $x_c^k \in \argmin\left\{f^k(x) \mid x\in \mathcal{P}^k\right\}$
\State $S^k \gets \left\{\begin{array}{l}
    Success~\textit{if}~f^k(x_c^k) < f^k(x_*^{k-1}) \\
    Barrier~\textit{if}~P^k \cap D_f = \emptyset \\
    Failure~\textit{otherwise}
\end{array}\right.$
\Return $x_c^k, S^k, \Cache^k$
\EndFunction
\end{algorithmic}
\end{algorithm}

Algorithm~\ref{algo:poll} implements the poll step using the mechanics from \mads, and returns:
\begin{itemize}
    \item $x_c^k$, the best candidate found during the search,
    \item $S^k$, an indicator of the quality of this candidate: $S^k=Success$ if it appears better than the incumbent, $S^k=Barrier$ if no point of $\P^k$ belongs to the feasible domain $D_f$, and $S^k=Failure$ if no feasible candidate have an estimate lower than $f^k(x_*^{k-1})$,
    \item $\Cache^k$, the former cache updated with the evaluations performed during the poll.
\end{itemize}

Observe that Line $8$ indicates to compute some values $(\sigma_i)_{1 \leq i \leq u}$ so that the standard deviation $\sigma^k(x)=\left(\frac{1}{\sigma^{k-1}(x)^2}+\sum_{i=1}^{u}\frac{1}{\sigma_i^2}\right)^{-1/2}$ of estimate $f^k(x)$ satisfies $\sigma^k(x) \leq \sigma^k$.
This means that the $Poll$ checks the quality of the estimate $f^{k-1}(x)$ by comparing its current standard deviation $\sigma^{k-1}(x)$ to $\sigma^k$. 
If it is higher than $\sigma^k$, then the algorithm produces one or many standard deviations $(\sigma_i)_{1 \leq i \leq u}$ and observations $\left(f_{\sigma_i}(x;\omega_i)\right)_{1 \leq i \leq u}$ such that the new estimate $f^k(x)$ have a standard deviation satisfying $\sigma^k(x) \leq \sigma^k$. 
A simple strategy consists of the following: if $\sigma^{k-1}(x)\leq\sigma^k$, no new observation is performed, but if $\sigma^{k-1}(x)>\sigma^k$ an unique ($u=1$) standard deviation is produced at an high value ($\sigma$ such that $\sigma^k(x)=\sigma^k$, or $\sigma=\sigma_{max}$ if the former equation leads to a solution $\sigma$ higher than $\sigma_{max}$).

\dpmads and \mpmads also allow the optional \guillemets{search step} from \mads to be used at the beginning of iteration $k$, with the $Search(\Cache^{k-1},r^k)$ function. 
It improves the estimates of all or some points $x$ in $\M^k$ with an observation of $f_{\sigma_s}(x)$ for a given $\sigma_s$. To avoid this step being too costly, this standard deviation $\sigma_s$ can be set at a high value. 
In the following, it is set to $\sigma_s^k=\rho(r^k-r_s)$ for a given $r_s>0$. 
$Search$ creates $\Cache^{k-1}_{search}$, the cache $\Cache^{k-1}$ updated with the addition of the estimates computed by the function, then returns a minimiser $x_s^k$ of $f^{k-1}$ over $\Cache^{k-1}_{search}$. If this step is disabled, $\Cache^{k-1}_{search}=\Cache^{k-1}$ and $x_s^k=x_*^{k-1}$ are unchanged. Altrough it does not impact convergence, it is possible that $x_s^{k-1}$ differs from $x_*^{k-1}$.

%===================================%
\subsection[MADS algorithm with monotonic and dynamic precision control]{\mads algorithm with monotonic and dynamic precision control}
\label{section:Algos}
%===================================%

The comprehension of the behaviour of both algorithms is easier if one recall how to decide if estimates are sufficiently accurate. At the end of iteration $k$, \dpmads and \mpmads consider their incumbent $x_*^k$ as a point which have the lowest objective function value estimate among the set of evaluated points: $x_*^k \in \argmin\{f^k(x) \mid \Cache^k(x) \neq \emptyset\}$. 
However, because of the noise, it is uncertain that this incumbent also minimises $f$ over the evaluated points. 
This uncertainty is quantified as follows. One can compare the best candidate $x_c^k \in \argmin\{f^k(x) \mid x\in\P^k\}$ to $x_s^k$ by computing the statistical \textit{p-value} $p^k$ of the hypothesis \guillemets{$f(x_c^k)<f(x_s^k)$}, knowing $\Cache^k$. 
If $p^k$ is close to $1$, then \guillemets{$f(x_c^k)<f(x_s^k)$} is highly plausible and the estimates $f^k(\cdot)$ are considered sufficiently accurate. 
If $p^k$ is close to $0$, the estimates are considered accurate because the opposite hypothesis \guillemets{$f(x_c^k)\geq f(x_s^k)$} is highly plausible. 
Then, the estimates are considered inaccurate if $p^k$ is too close to $0.5$. 
Figure~\ref{fig:Comparison} illustrates the hypothesis \guillemets{$f(x_2)<f(x_1)$} on three scenarios.
The values $p^k$ are computed analytically, using the cumulative distribution function of the $\Normale{0}{1}$ law denoted $\Phi$.

\begin{figure}[!ht]
\centering
\includegraphics[width=0.9\linewidth]{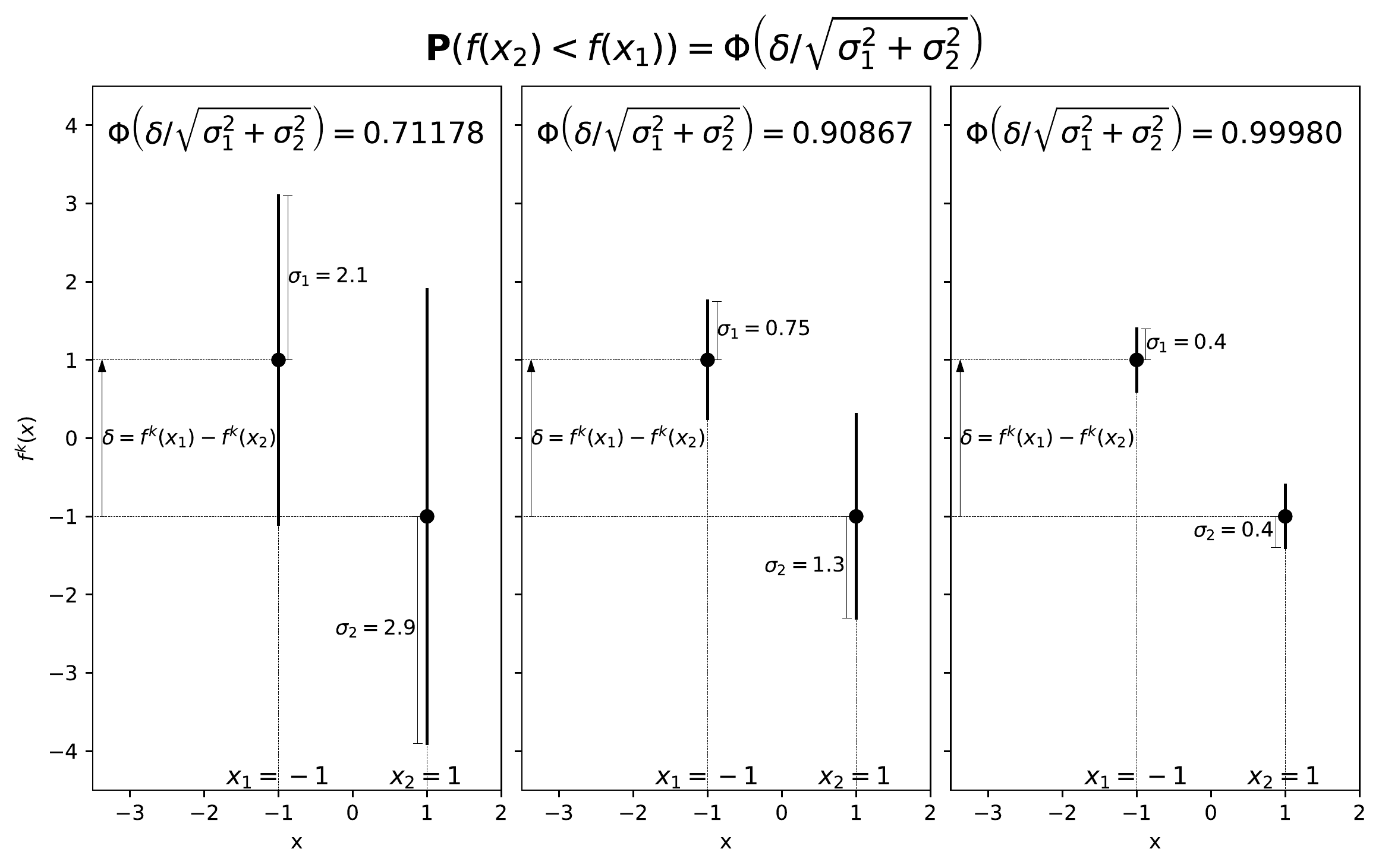}
\caption{Comparison of two estimates in three situations. The dots represent the estimated objective function values and the vertical lines the standard deviations.}
\label{fig:Comparison}
\end{figure}

\mpmads is designed to avoid the situations where an assumption is not highly plausible. 
It improves its precision index $r^k$ if $p^k$ or $1-p^k$ lies in the interval $[\beta_\ell,\beta_u]=[0.03\%,99.7\%]$. 
\dpmads is more tolerant, as it also considers the computational cost required to reach low standard deviations. 
It improves the precision if $p^k$ or $1-p^k$ is in $[\beta_\ell,\beta_u]=[15\%,85\%]$ but can reduce it if plausibility becomes too high (considering that plausibility of, say, $99.7\%$ is more than strictly required to avoid biased convergence, because $90\%$ plausibility is almost as viable and is cheaper to reach).
With these precision ranges, \mpmads only accepts the hypothesis \guillemets{$f(x_2)<f(x_1)$} on the rightmost scenario of Figure~\ref{fig:Comparison}, while \dpmads accepts it on the two rightmost ones.
These thresholds are chosen in concordance with the computation of uncertainty performed by the p-value: the difference $\delta$ between two estimates $f^k(x)$ and $f^k(y)$ is seen as an observation of an equivalent normal law centred on $0$ and with standard deviation $\sigma=\sqrt{\sigma^k(x)^2+\sigma^k(y)^2}$, and $p^k$ is the probability to observe $x \leq \delta$ within this law. With the chosen thresholds, \mpmads accepts a new incumbent only when $\abs{\delta} \geq 3\sigma$, while \dpmads accepts it when $\abs{\delta} \geq \sigma$.

Let also $C_\textsc{Dp}$ and $C_\textsc{Mp}$ be two logical conditions about the precision index $r^k$:
\begin{equation}\label{eq:minimalconditions}
C_\textsc{Dp}:~\left\{p^k \in [\beta_\ell,\beta_u] \implies r^{k+1} > r^k\right\}
~\text{and}~
C_\textsc{Mp}:~\left\{\begin{array}{c}
    C_\textsc{Dp}~\textit{is satisfied} \\
    p^k \notin [\beta_\ell,\beta_u] \implies r^{k+1} = r^k
\end{array}\right\}. \end{equation}
The first one, $C_\textsc{Dp}$, is the minimal requirement to make the algorithms working. 
It forces the precision index $r^k$ to increase as soon as the indicator of the quality of the estimates ($p^k$) shows that the estimates are not precise enough. 
However, it does not dictate any behaviour in the other case $p^k \notin [\beta_\ell,\beta_u]$. Then, the precision could either increase, decrease or remain constant in this situation. 
The second condition, $C_\textsc{Mp}$, is more stringent. 
It also imposes to $r^k$ to strictly increase if $p^k \in [\beta_\ell,\beta_u]$ but in addition, it forces $r^k$ to remain constant in the other situation. 
Therefore, under the condition $C_\textsc{Mp}$ the precision grows monotonically during the optimization process, while it is not necessarily the case when only $C_\textsc{Dp}$ is satisfied.

The only difference between \mpmads and \dpmads relies on the $UpdateR(r^k,p^k)$ function which leads the evolution of the precision index $r$. 
At any iteration, in \dpmads the function $UpdateR$ only needs to satisfy $C_\textsc{Dp}$, 
    while in \mpmads it needs to satisfy $C_\textsc{Mp}$.
Therefore, \mpmads is a specific variant of \dpmads.
With both algorithms the precision $r^k$ is forced to increase as soon as the p-value $p^k$ belongs to $[\beta_\ell,\beta_u]$. 
When $p^k$ lies outside this range, \mpmads forces $r^{k+1}$ to remain unchanged (equal to $r^k$), while in \dpmads the precision can also be increased, or even reduced.
In other words, in \mpmads the precision index remains constant until it is uncertain that the candidate $x_c^k$ is better than $x_s^k$ or not. 
In \dpmads, regardless of the certainty of this comparison, the precision index can either increase or decrease. 
It increases, not necessarily by $r^{k+1}=r^k+1$, if $p^k$ is too close to $0.5$,
 and decreases if it is too far from $0.5$.
Also, default values of $\beta_\ell$ and $\beta_u$
 are not the same on both algorithms.
\mpmads uses restrictive thresholds ($0.03\%$ and $99.7\%$) while in \dpmads there is more flexibility
 ($15\%$ and $85\%$ as default values).

\dpmads and \mpmads are formulated under these notations and concepts in Algorithm~\ref{algo:dp/mp-mads}.

\begin{algorithm}[ht]
\caption{\dpmads and \mpmads structure}\label{algo:dp/mp-mads}
\begin{algorithmic}[1]
\Function{\mads}{$f$; $x_*^0$; $\rho$, $\beta_\ell,~\beta_u$}
\State $Initialise: r^0=0,\Cache^0=\emptyset,\delta_p^0=1,k=1$
\While{\Not{$\textit{Stopping criteria reached}$}}
    \State $x_{s}^k,\Cache_{search}^{k-1} \gets Search(\Cache^{k-1},r^k)$
    \State $x_c^k, S^k, \Cache^k \gets Poll\left(x_s^k, \delta_p^k, r^k, \Cache_{search}^{k-1}\right)$
    \State $p^k \gets \Call{PValue}{x_c^k,x_s^k,\Cache^k}$
    \If {$S^k = Success$}
        \State $\delta_p^{k+1} \gets \left\{\begin{array}{rl}
            2\delta_p^k & \textit{if}~p^k>\beta_u \\
             \delta_p^k  & \textit{otherwise}
        \end{array}\right.$
        \State $r^{k+1} \gets UpdateR(r^k,p^k),~\textit{satisfying}~C_\textsc{Mp}~\textit{for}~\mpmads,~\textit{or}~C_\textsc{Dp}~\textit{for}~\dpmads$
    \ElsIf {$S^k = Failure$}
        \State $\delta_p^{k+1} \gets \left\{\begin{array}{ll}
            \delta_p^k/2 & \textit{if}~p^k<\beta_\ell \\
            \delta_p^k   & \textit{otherwise}
        \end{array}\right.$
        \State $r^{k+1} \gets UpdateR(r^k,p^k),~\textit{satisfying}~C_\textsc{Mp}~\textit{for}~\mpmads,~\textit{or}~C_\textsc{Dp}~\textit{for}~\dpmads$
    \Else~($S^k=Barrier$)
        \State $\delta_p^k \gets \delta_p^k/2$
        \State $r^{k+1} \gets r^k$
    \EndIf
    \State $x_*^k \in \argmin_{y\in \P^k}(f^k(y))$
    \State $k \gets k+1$
\EndWhile
\Return $x_*^k$
\EndFunction
\end{algorithmic}
\end{algorithm}

%===================================%
\subsection{Practical implementations of the \dpmads and \mpmads algorithms}
\label{section:Implementation}
%===================================%

This section gives additional information necessary to implement the algorithms from Section~\ref{section:Algos}. 
As the adaptive precision program may be unable to propose an arbitrarily high standard deviation, 
    then one may define $\sigma_{max}=\underset{\sigma\geq0}{\sup}\{\sigma \mid f_\sigma(x)~\textit{can be observed}\}$.
One can also propose $\sigma_{min}>0$ if one does not want the algorithms to ask for arbitrarily costly computation. 
Then, $\rho$ has to satisfy $\underset{r\rightarrow-\infty}{\lim}\rho(r)=\sigma_{max}$ and $\underset{r\rightarrow+\infty}{\lim}\rho(r)=\sigma_{min}$. 
A parameter $\theta>0$ is proposed to control the decrease rate
    (it can be seen as the attenuation of the noise magnitude, given in decibel). 
Also, a reference index $r_0$ is defined, such that $\rho(r_0)=\frac{\sigma_{min}+\sigma_{max}}{2}$. 
Then, the $\rho$ function is given by
\begin{equation*}\label{eq:rho}
    \rho:\left\{\begin{array}{rcl}
        \R & \longrightarrow & [\sigma_{min};\sigma_{max}] \\
        r  & \longmapsto     & \left\{\begin{array}{ll}
            \sigma_{min}+\dfrac{\sigma_{max}-\sigma_{min}}{2}10^{-(r-r_0)\theta} & \textit{if}~r\geq r_0 \\
            \sigma_{min}+\dfrac{\sigma_{max}-\sigma_{min}}{2}\left(2-10^{(r-r_0)\theta}\right) & \textit{if}~r<r_0.
        \end{array}\right.
    \end{array}\right.
\end{equation*}
and is represented in Figure~\ref{fig:Rho}. 
Default values are $\sigma_{max}=1, \sigma_{min}=0, r_0=0$.

\begin{figure}[!ht]
\centering
\includegraphics[width=0.75\linewidth]{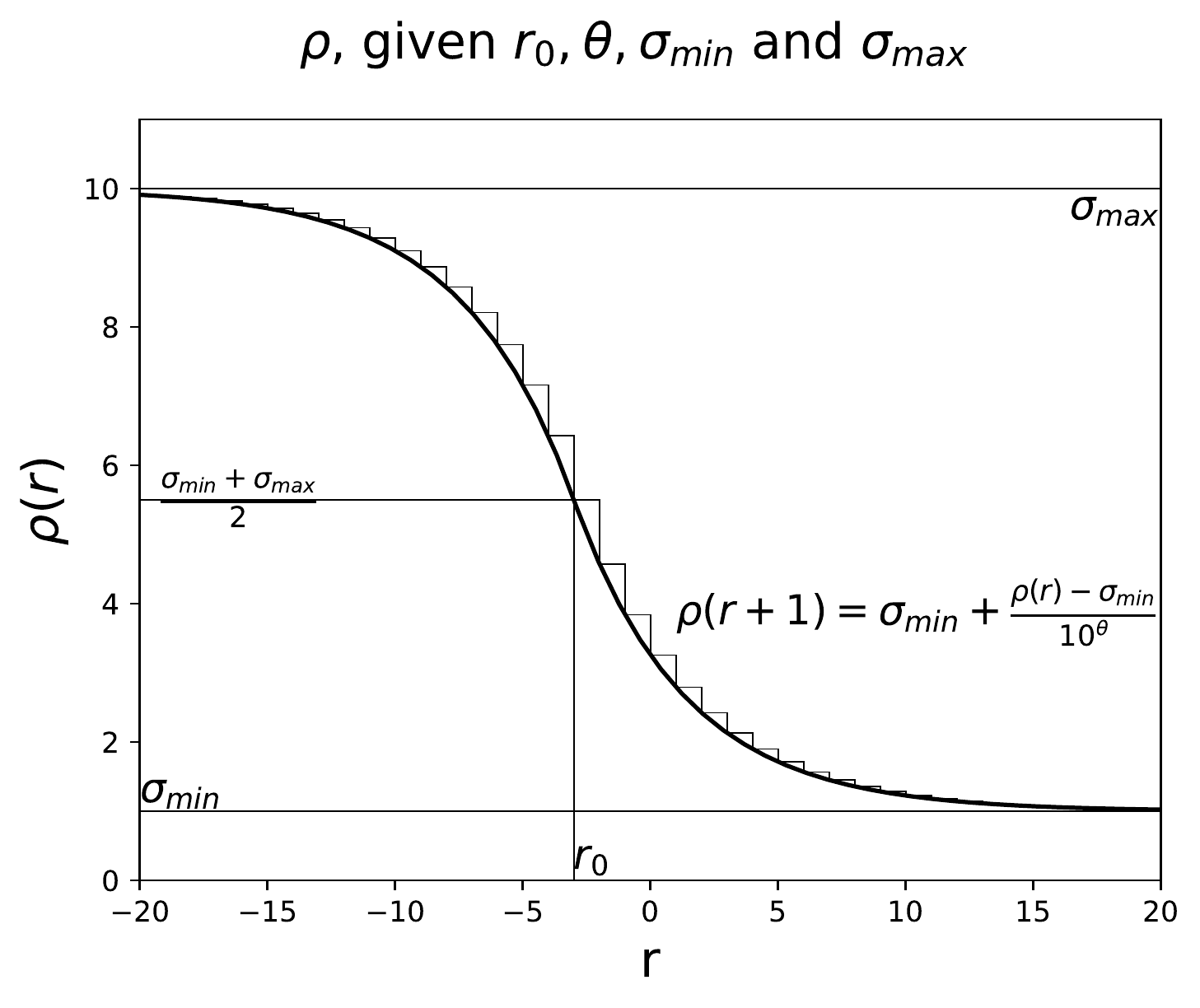}
\caption{The $\rho$ function with parameters $r_0=-3$, $\theta=0.1$, $\sigma_{min}=1$, $\sigma_{max}=10$.}
\label{fig:Rho}
\end{figure}

The optional function $Search$ is disabled for \mpmads. 
In \dpmads, it is called with the internal parameter $r_s=-5$, and re-estimates only the points which have an high enough plausibility to appear better than the incumbent. In other words, \dpmads's $Search$ function at iteration $k$ uses an observation with $\sigma_s=\rho(r^k-r_s)$ to update the estimates of all the points in the following set: $$ \left\{x\in\Cache^{k-1} \mid PValue\left(x,x_*^{k-1},\Cache^{k-1}\right) \geq \tau\right\}~\text{(with by default}~\tau=0.25\text{)}.$$

The last function needing to be described is the $UpdateR$ function. 
For \dpmads, the only theoretical requirement is $C_\textsc{Dp}: r^{k+1}=UpdateR(r^k,p^k)>r^k$ if $p^k \in [\beta_\ell,\beta_u]$. 
An acceptable function is represented in Figure~\ref{fig:UpdateR}, using various thresholds on $p$ to control the variations on $r$. For \mpmads, $r^{k+1} \neq r^k \iff p^k \in [\beta_\ell,\beta_u]$ and if so, $r^{k+1} > r^k$ is mandatory. In Figure~\ref{fig:UpdateR}, the proposed function computes $r^{k+1}=r^k+1$ as soon as $p^k \in [\beta_\ell,\beta_u]$, regardless of its value. However, some thresholds could be proposed to allow a non-unitary increasing of the precision index.
\begin{figure}[!ht]
\centering\includegraphics[width=0.8\linewidth]{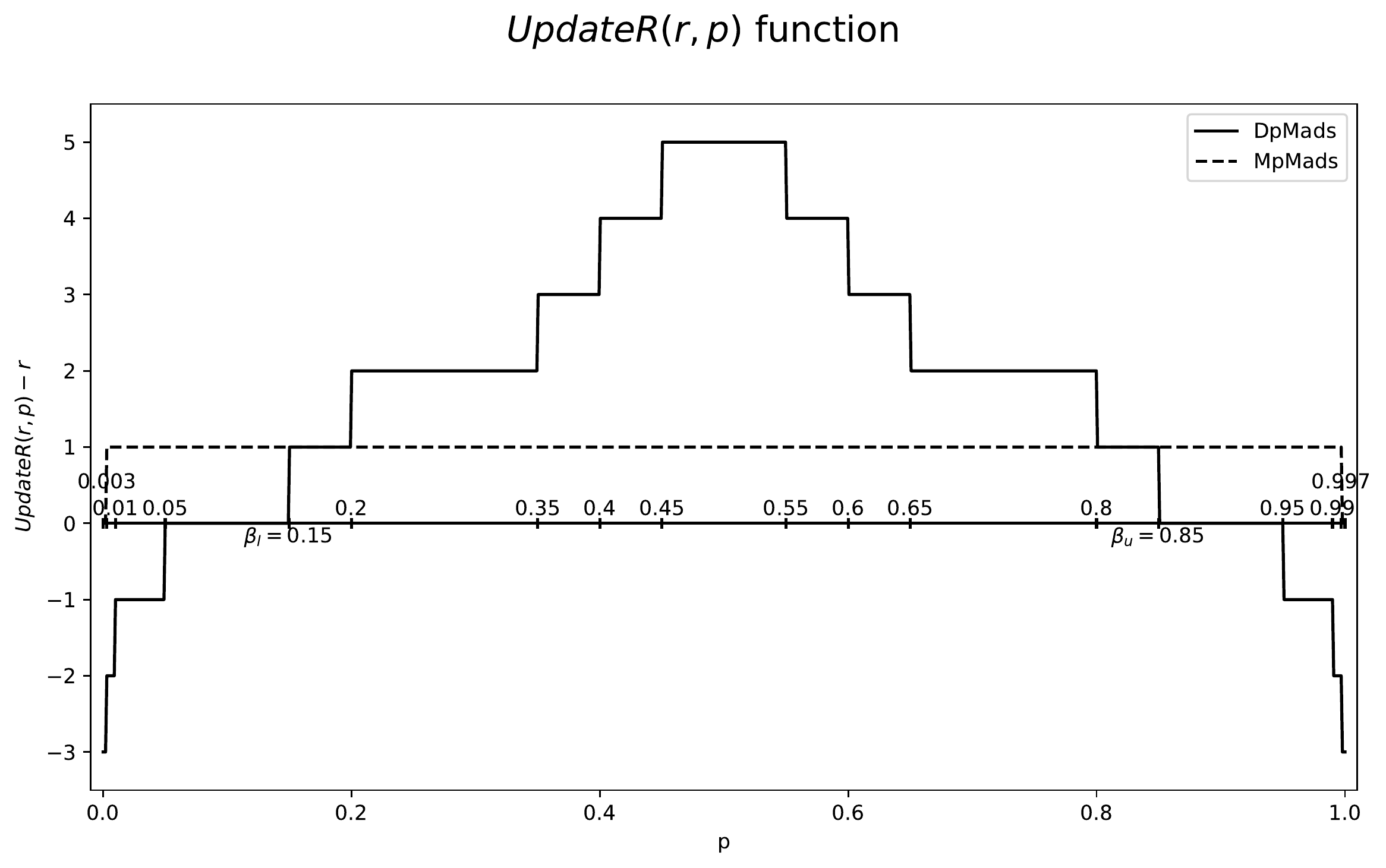}
\caption{Default $UpdateR$ function, given $\beta_\ell$, $\beta_u$ and some thresholds.}
\label{fig:UpdateR}
\end{figure}

Modifying these practical parameters requires precaution. 
Some values of the $\sigma_{min}$ parameter may lead the algorithms to fail it the $Search$ function is disabled. Due to the algorithmic mechanics, an estimate $f^k(x)$ satisfying $\sigma^k(x)<\sigma^k$ is not re-evaluated by the $Poll$ step.
If the $Search$ is enabled, it will perform such a re-evaluation. However, if it is disabled, the $Poll$ step cannot reduce the standard deviation $\sigma^k(\cdot)$ lower than $\sigma_{min}$, making the convergence impossible to achieve.
Then, $\sigma_{min}$ have to be forced to $0$ if the $Search$ is disabled.
This remark is especially important for \mpmads, as the $Search$ is disabled by default. Also, the parameter $\beta_\ell$ has to be strictly positive: $0 < \beta_\ell \leq 0.5$.

%===================================%
%===================================%
\section{Convergence analysis}
\label{section:Proof}
%===================================%

This section studies the convergence of both algorithms over the
Problem~\eqref{eq:OptimProblem}, under an additional assumption that $f$ is defined everywhere: $D_f=\R^n$, and has all its level sets bounded. 
It is also assumed that $\sigma_{min}=0$ and that the sequence of precision indexes $(r^k)_k$ satisfies the condition $C_\textsc{Dp}$ from~\eqref{eq:minimalconditions}.
The main idea of the proof is that in the absence of a stopping criteria, 
    the algorithms generates a sequence of estimated optima $x_*^k$, 
    with an accumulation point denoted $x_*$. 
This point almost surely satisfies local conditions based on the Clarke derivatives (defined in~\cite{Clar83a}) of the true objective function $f$.

%===================================%
\subsection{Technical lemmas}
%===================================%

This section gives some useful technical results.
The first one defines an optimization problem to provide an upper bound for a quantity which appears in the proof of Theorem~\ref{theorem:UpperBoundLevelSet}.
%-----------------------------------%
\begin{lemma}[Maximum of a sum of products of variables]\label{lemma:ProblemSumProd}
Let $n\geq2$ be an integer and $M$ the set of $n\times n$ matrices with all entries in the interval $[0,1]$. Denote $\Phi_{i,j} \in M$ such a matrix (this notation is chosen in concordance with the proof of Theorem~\ref{theorem:UpperBoundLevelSet}, on which some cumulative distribution of the centred reduced normal law are calculated). The problem
\begin{equation*}\begin{array}{rl}
    \underset{\Phi\in M}{\max} & \dsum{j=1}{n}\dprod{i=1}{n}\Phi_{i,j} \\
    \text{subject to} & \left\{\begin{array}{rcll}
                        \Phi_{i,i} &=& 1 & \forall i\in\ll1;n\rr \\
                        \Phi_{i,j} &=& 1-\Phi_{j,i} & \forall i,j \in \ll1;n\rr^2 \mid i>j.
                        \end{array}\right.
\end{array}\end{equation*}
has an optimal objective value equal to $1$.
\end{lemma}
%-----------------------------------%

%-----------------------------------%
\begin{proof}
Denote $E=\{(i,j) \in \N^2 \mid 1 \leq i < j \leq n\}$ and $f$ the objective function. 
The variables $\Phi_{i,j}$ with $(i,j)\notin E$ can be removed from the problem using the equality constraints, 
    leading to a bound-constrained reformulation:
$$ \underset{\Phi_{i,j}\in[0,1]~\forall(i,j)\in E}{\max}~f(\Phi)=\dsum{j=1}{n}\dprod{i<j}{}\Phi_{i,j}\dprod{i>j}{}1-\Phi_{j,i}. $$

Denote also $\Pi_j=\dprod{i<j}{}\Phi_{i,j}\dprod{i>j}{}1-\Phi_{j,i},~\forall j \in \ll1;n\rr$. The problem is therefore
$$ \underset{\Phi_{i,j}\in[0,1]~\forall(i,j)\in E}{\max}~\dsum{j=1}{n}\Pi_j. $$

First, assume all the variables are either $0$ or $1$. If any of the $\Pi_j=1$ (say, $\Pi_b=1$), all the other $\Pi_j, j \neq b$, are necessarily $0$:
$$\begin{array}{rcl}
\Pi_b=1 & \Rightarrow & \Phi_{i,b}=1~\forall i<b~\textit{and}~\Phi_{b,i}=0~\forall i>b \\
          & \Rightarrow & \Pi_j = 0~\forall j \neq b.
\end{array} $$
This proves that if all the variables are either $0$ or $1$, objective value is at most $1$.
\comment{Two solutions leads to this value: $\{\forall (i,j) \in E,~\Phi_{i,j}=0\}$ and $\{\forall (i,j) \in E,~\Phi_{i,j}=1\}$.}

Second, one can prove that a solution with some variables in $]0,1[$ cannot be optimal. Denoting $\pi_{i,j}=\dprod{k<j,k \neq i}{}\Phi_{k,j}\dprod{k>j,k \neq i}{}1-\Phi_{j,k},~\forall i,j \in \ll1;n\rr^2$, the partial derivatives of $f$ are:
$$ \pdiff{f}{\Phi_{i,j}} = \pi_{j,i}-\pi_{i,j}. $$
This does not depend on $\Phi_{i,j}$, so $f$ is necessarily not optimal until every variables $\Phi_{i,j}$ are set to one of their bounds ($0$ or $1$ depending on the sign of $\pi_{j,i}-\pi_{i,j}$).
\end{proof}
%-----------------------------------%

The two following Lemma~\ref{lemma:ConsistentMinimiser} and~\ref{lemma:ConsistentEstimates} ensure that the estimated values $f^k$ of the function $f$, over a finite set, respect the partial order defined by $f$:
%-----------------------------------%
\begin{lemma}[Localisation of a minimum is consistent on finite sets]\label{lemma:ConsistentMinimiser}
Let $E$ be a set with a finite number of elements. 
Let $f: E\rightarrow\R$ be a function and $\forall x\in E,f^i(x)$ the maximum likelihood value of $f(x)$ constructed from 
 the set $\Cache^i(x)$, under the assumption that it contains $i$ elements. 
If $x_*\in\underset{x\in E}{\argmin}~f$ is the unique minimiser of $f$ on $E$,
    then
\begin{equation*} \Pr\left(\exists I\in\N: \forall i\geq I,~x_*\in\underset{x\in E}{\argmin}~f^i(x)\right)=1. \end{equation*}
\end{lemma}
%-----------------------------------%

%-----------------------------------%
\begin{proof}
Denoting $E=\{x_1,\dots,x_e,x_*\}$ (thus making $\#E=e+1$), one defines the optimality gaps $\mu_i=f(x_i)-f(x_*)>0,~\forall i \in \ll1;e\rr$.

The strong law of large numbers ensures the convergence of any estimate to the true value it approximates: $\forall x\in E,~\Pr\left(f^k(x)\underset{k\rightarrow+\infty}{\longrightarrow}f(x)\right)=1$. It follows that
\begin{equation*}\setlength{\abovedisplayskip}{3pt}\left\{\begin{array}{rl}
    \forall i\in \ll1;e\rr, & \Pr\left(\exists K_i: \forall k\geq K_i,~f^k(x_i)>f(x_i)-\mu_i/2\right)=1, \\
    \text{with }\eta=\underset{i\in \ll1;e\rr}{\min}\left\{\mu_i\right\}, & \Pr\left(\exists K_\eta: \forall k \geq K_\eta,~f^k(x_*)<f(x_*)+\eta/2\right)=1.
\end{array}\right.\setlength{\belowdisplayskip}{3pt}\end{equation*}
Then, the constant $K=\max\left\{K_1,\dots,K_e,K_\eta\right\}$ exists almost surely (as a maximum of a finite number of constants which all exists almost surely) and satisfies:
\begin{equation*}\setlength{\abovedisplayskip}{3pt}
    \forall i\in \ll1;e\rr,~\forall k\geq K,~f^k(x_i)>f(x_i)-\mu_i/2\geq f(x_*)+\eta/2>f^k(x_*)
\setlength{\belowdisplayskip}{3pt}\end{equation*}
which is equivalent to the result claimed by the Lemma:
\begin{equation*}\setlength{\abovedisplayskip}{3pt}
    \forall k\geq K,~x_*\in\underset{x\in E}{\argmin}f^k(x).
\setlength{\belowdisplayskip}{3pt}\end{equation*}
\end{proof}
%-----------------------------------%

Lemma~\ref{lemma:ConsistentMinimiser} can be extended if the images by the deterministic function $f$ are all different:
%-----------------------------------%
\begin{corollary}[Partial orders defined by the estimates $f^k(\cdot)$ are consistent on finite sets]\label{lemma:ConsistentEstimates}
With the notations from Lemma~\ref{lemma:ConsistentMinimiser}, assume the images $f(x)$ are all different (in the sense that $\forall (x,y) \in E^2 \mid y \neq x,~f(x) \neq f(y)$). The estimates $f^k$ eventually defines a coherent partial-order relation on $E$:
\begin{equation*} \Pr\left(\exists K : \forall k\geq K,~\forall x,y\in E^2,~f(x)<f(y) \iff f^k(x)<f^k(y)\right)=1. \end{equation*}
\end{corollary}
%-----------------------------------%

%-----------------------------------%
\begin{proof}
Let $E = \{x_1,\dots,x_{e+1}\}$ with the elements $x_i$ ordered as $f(x_1)\leq f(x_2)\leq\dots\leq f(x_{e+1})$. One can recursively apply Lemma~\ref{lemma:ConsistentMinimiser} to $E$, then $E\setminus\{x_1\}$, then $E\setminus\{x_1,x_2\}$, \dots, until $E\setminus\{x_1,x_2,\dots,x_e\}$. One obtains some constants denoted $K_0,K_1,\dots,K_e$. Then, $K=\dsum{i=0}{e}K_i$ satisfies the desired definition (because $K_0$ iterations makes $x_1$ to be the minimiser over $f^k$, then $K_1$ makes $x_2$ to minimise the remaining $f^k$, and so on).
\end{proof}
%-----------------------------------%

Notice that neither Lemma~\ref{lemma:ConsistentMinimiser} nor Corollary~\ref{lemma:ConsistentEstimates} ensures anything for points $x$ and $y$ such that $f(x)=f(y)$. The reason is when two points have the exact same image by $f$, the estimates $f^k(x)$ and $f^k(y)$ could satisfy either $f^k(x)<f^k(y)$ or $f^k(x)>f^k(y)$ for any iteration $k$. Estimates comparisons are therefore irrelevant when $f(x)=f(y)$.

It follows from Lemma~\ref{lemma:ConsistentMinimiser} that if two points $x$ and $y$ are evaluated infinitely often, the p-value of hypothesis \guillemets{$f(x)<f(y)$} knowing $\Cache^k$ will converge to $0$ (if $f(x)<f(y)$) or $1$ (if $f(x)>f(y)$).
%-----------------------------------%
\begin{lemma}[Behaviour of the p-value used to compare points]\label{lemma:PValueLimit}
Let $x$ and $y$ be two points satisfying $f(x)\neq f(y)$. 
Let $f^k(\cdot)$ be the maximum likelihood estimate of $f(\cdot)$ constructed with $k$ observations: 
    $\#\Cache^k(\cdot)=k$. 
Let $p^k$ be the p-value of the hypothesis \guillemets{$f(x)<f(y)$}, knowing $\Cache^k$. This p-value satisfies:
\begin{equation*}\left\{\begin{array}{rccl}
        f(x)<f(y) & \Rightarrow & p^k & \underset{k\rightarrow+\infty}{\longrightarrow}1, \\
        f(x)>f(y) & \Rightarrow & p^k & \underset{k\rightarrow+\infty}{\longrightarrow}0.
\end{array}\right.\end{equation*}
\end{lemma}
%-----------------------------------%

%-----------------------------------%
\begin{proof}
Assume, without loss of generality, that $f(y)>f(x)$. let $\delta=(f(y)-f(x))/2$.
Lemma~\ref{lemma:ConsistentMinimiser} ensures that $\exists K \mid \forall k \geq K,~f^k(y)-f^k(x)>\delta$. Therefore, denoting $\Phi$ the cumulative distribution function of the law $\Normale{0}{1}$, $$\forall k \geq K,~p^k=\Phi\left(\dfrac{f^k(x)-f^k(y)}{\sqrt{\sigma^k(x)^2+\sigma^k(y)^2}}\right)<\Phi\left(\dfrac{-\delta}{\sqrt{\sigma^k(x)^2+\sigma^k(y)^2}}\right).$$
Also, denote $(\sigma_1,\dots,\sigma_k)$ the $k$-standard deviations of the $k$ estimates of $f(x)$. As any of those $\sigma_i$ is at most $\sigma_{max}=\underset{r\rightarrow-\infty}{\lim}\rho(r)$, one have: $$\sigma^k(x)=\sqrt{\dfrac{1}{\dsum{i=1}{k}1/\sigma_i^2}} \leq \sqrt{\dfrac{\sigma_{max}^2}{k}}.$$
With a similar argument, $\sigma^k(y) \leq \sqrt{\dfrac{\sigma_{max}^2}{k}}$. Therefore,
$$ p^k \leq \Phi\left(\dfrac{-\delta \sqrt{k}}{\sqrt{2}\sigma_{max}}\right) \underset{k\rightarrow+\infty}{\longrightarrow}0, $$ which concludes the proof.
\end{proof}
%-----------------------------------%

In the specific case of two different points $x$ and $y$ such that $f(x)=f(y)$, the behaviour of $p^k$ is different. Lemma~\ref{lemma:PValueLimitPtsEgaux} describes it:
%-----------------------------------%
\begin{lemma}[Behaviour of the p-value for points with identical image]\label{lemma:PValueLimitPtsEgaux}
With notations from Lemma~\ref{lemma:PValueLimit} and assumption $f(x)=f(y)$, $p^k$ statistically satisfies:
$$ \forall \varepsilon \in [0,0.5],~\Pr\left(p^k>0.5+\varepsilon\right)=\Pr\left(p^k<0.5-\varepsilon\right)=0.5-\varepsilon. $$
\end{lemma}

%-----------------------------------%
\begin{proof}
The estimates $f^k(\cdot)$ statistically follow normal laws $\Normale{f(\cdot)}{\sigma^k(\cdot)^2}$. As $f^k(x)$ and $f^k(y)$ are independent, their reduced difference follows $\dfrac{f^k(x)-f^k(y)}{\sqrt{\sigma^k(x)^2+\sigma^k(y)^2}} \sim \Normale{0}{1}$.

Also, recall that $p^k=\Phi\left(\dfrac{f^k(x)-f^k(y)}{\sqrt{\sigma^k(x)^2+\sigma^k(y)^2}}\right)$. As such, at any iteration $k$ its expected value is $0.5$ but it can reach any value in $[0,1]$. This ensures the result, because:
$$ \forall \varepsilon \in [0,0.5],~\left\{\begin{array}{rclcl}
    \Pr\left(p^k>0.5+\varepsilon\right) & = & \Pr\left(\Normale{0}{1}>\Phi^{-1}(0.5+\varepsilon)\right) & = & 0.5-\varepsilon, \\
    \Pr\left(p^k<0.5-\varepsilon\right) & = & \Pr\left(\Normale{0}{1}<\Phi^{-1}(0.5-\varepsilon)\right) & = & 0.5-\varepsilon.
\end{array}\right.
 $$
\end{proof}
%-----------------------------------%

A crucial requirement of the \mads algorithm is that all generated trial points lie on the mesh $\M^k$.
This assumption leads to the following lemma, which ensures that at any iteration $k$, the generated optimum $x_*^k$ and any point evaluated by the algorithms lie on a given mesh.
%-----------------------------------%
\begin{lemma}[All points evaluated up to iteration $k$ lie on a given mesh]\label{lemma:MinimalMesh}
Denote, for iteration $k$, the smallest mesh parameter encountered up to iteration $k$ by $\delta_{min}^k=\underset{i\leq k}{\min}\left\{\delta_p^i\right\}$. Then, the incumbent $x_*^k$, and any point evaluated $(\{x\mid\Cache^k(x)\neq\emptyset\})$ lie on a given mesh:
\begin{equation*} \forall x\in\R^n\mid\Cache^k(x)\neq\emptyset,~x\in\M_{\delta_{min}^k}\left(x_*^0\right). \end{equation*}
\end{lemma}
%-----------------------------------%

%-----------------------------------%
\begin{proof}
The fact that $x_*^0\in\M_{\delta_{min}^0}\left(x_*^0\right)$ holds trivially. Recursively, assume $x_*^k$ and all the points generated up to iteration $k$ are on the mesh. Remind also that $\delta_p^{k+1}$ is, by construction, on the form $2^{(s^k)}\delta_{min}^k$ with $s^k\in\N$. One can observe that iteration $k+1$ evaluates only $x_*^k$ (which is assumed to be on the mesh), the candidates coming from $\P^{k+1}$, and potentially other points generated by the $Search$ step (which all are, by construction, elements of the larger mesh $\M^{k+1} = \M_{\delta_m^{k+1}}(x_*^k) \subseteq \M_{\delta_{min}^{k+1}}\left(x_*^k\right)$).
\end{proof}
%-----------------------------------%

%===================================%
\subsection{Convergence theorems}
%===================================%

The next result shows that the sequence of incumbents remains bounded if all the level sets of $f$ are bounded.
%-----------------------------------%
\begin{theorem}[A bounded level set contains every visited points]\label{theorem:UpperBoundLevelSet}
Assume that all level sets of the objective function $f$ are bounded.
%Let the objective function $f$ having all its level sets bounded. 
Let $x_*^0$ be the feasible starting point of an execution of one of the algorithms, and $\left(x_*^k\right)_k$ be the sequence of estimates generated during the optimization process. 
There exists, with probability one, a ball centred on $x_*^0$ which contains the entire sequence $\left(x_*^k\right)_k$:
\begin{equation*} \Pr\left(\exists R>0: \forall k\in\N,~x_*^k\in B_R\left(x_*^0\right)\right)=1. \end{equation*}
\end{theorem}
%-----------------------------------%

%-----------------------------------%
\begin{proof}
The proof exploits the Borel-Cantelli Lemma, which provides:
\begin{quote}[Borel-Cantelli Lemma]\label{lemma:BorelCantelli}
Let $(A^i)_{i\in\N}$ be a sequence of events. There is a condition to determine if all but a finite number of these events are realised:
\begin{equation*}\setlength{\abovedisplayskip}{3pt}
    \dsum{i=0}{\infty}\Pr\left(A^i\right)<+\infty \implies \Pr\left(A^i~\text{infinitely often}\right)=0.
\setlength{\belowdisplayskip}{3pt}\end{equation*}
\end{quote}

Since $f$ has bounded level sets $L_i=\{x\in\R^n\mid f(x)\leq f(x_*^0)+i\}$, 
one can define bounds $d_i$  so that $L_i \subseteq B_{d_i}(x_*^0), \forall i\in\N$.
Applying the Borel-Cantelli lemma to the following sequence of events:
\begin{equation*}\setlength{\abovedisplayskip}{3pt}\begin{array}{rcl}
    \forall i\in\N,~A^i & = & \{(x_*^k)_k~\text{has an element outside of}~L_i\} \\
     & = & \{\kappa(i)<+\infty\},~\text{with $\kappa(i)$ the lowest index $k$ such that $x_*^k \notin L_i$},
\end{array}\setlength{\belowdisplayskip}{3pt}\end{equation*}
it is possible to prove that all but a finite number of these events happens, almost surely. Then, denoting $\ell$ the index of the last $A^i$ which is realised, the ball $B_{d_\ell}(x_*^0)$ contains the whole sequence of incumbents. To show that $\sum_{i=0}^{\infty}\Pr(A^i)$ converges, one may compute the probability that the $k^{th}$ incumbent is outside of $L_i$, knowing the cache $\Cache^k$:
\begin{equation*}\setlength{\abovedisplayskip}{3pt}\begin{array}{rcl}
\Pr\left(x_*^k \notin L_i \mid \Cache^k\right)
&=& \E_{\Cache^k}\left(\1_{L_i^c}\left(x_*^k\right)\right) \\
&=& \dsum{x\in\Cache^k}{}\1_{L_i^c}\left(x\right)\Pr\left(x_*^k=x \mid \Cache^k\right) \\
&=& \dsum{x\in\Cache^k}{}\1_{L_i^c}\left(x\right)\Pr\left(\forall y\in\Cache^k,~f^k(y)\geq f^k(x) \mid \Cache^k\right) \\
&=& \dsum{x\in\Cache^k}{}\1_{L_i^c}\left(x\right)\dprod{y\in\{x,x_*^0\}}{}\Pr\left(f^k(y) \geq f^k(x) \mid \Cache^k\right)\dprod{y\in\Cache^k\setminus\{x,x_*^0\}}{}\Pr\left(f^k(y)\geq f^k(x) \mid \Cache^k\right) \\
&=& \dsum{x\in\Cache^k}{}\1_{L_i^c}\left(x\right)\Phi\left(\dfrac{f(x_*^0)-f(x)}{\sqrt{\sigma^k(x_*^0)^2+\sigma^k(x)^2}}\right)\dprod{y\in\Cache^k\setminus \{x,x_*^0\}}{}\Phi\left(\dfrac{f(y)-f(x)}{\sqrt{\sigma^k(y)^2+\sigma^k(x)^2}}\right).
\end{array}\setlength{\belowdisplayskip}{3pt}\end{equation*}

Any point $x$ for which $\1_{L_i^c}$ is non-zero satisfies $f(x) \geq f(x_*^0)+i$, and all the evaluated points satisfies $\sigma^k(\cdot) \leq \sigma_{max}$. 
Then, denoting $\Cache^k_r=\Cache^k\setminus\{x_*^0\}$: 
\begin{equation*}\setlength{\abovedisplayskip}{3pt}\begin{array}{rcl}
\Pr\left(x_*^k \notin L_i \mid \Cache^k\right)
&\leq& \dsum{x\in\Cache_r^k}{}\1_{L_i^c}\left(x\right)\Phi\left(\dfrac{-i}{\sqrt{2\sigma_{max}^2}}\right)\dprod{y\in\Cache_r^k\setminus \{x\}}{}\Phi\left(\dfrac{f(y)-f(x)}{\sqrt{\sigma^k(y)^2+\sigma^k(x)^2}}\right).
\end{array}\setlength{\belowdisplayskip}{3pt}\end{equation*}

If all the points in $\Cache_r^k$ are in $L_i$, this expression is zero. 
Otherwise, let $v^k$ be the number of points belonging to the cache which are outside of $L_i$ and denote them by $\{x_1,\dots,x_{v^k}\} = \Cache_r^k\cap L_i^c$.

Denote $\Phi_{i,j}=\Phi\left(\dfrac{f(x_i)-f(x_j)}{\sqrt{\sigma^k(x_i)^2+\sigma^k(x_j)^2}}\right)$. 
The previous expression can be written as:
\begin{equation*}\setlength{\abovedisplayskip}{3pt}\begin{array}{rcl}
\Pr\left(x_*^k \notin L_i \mid \Cache^k\right)
&\leq& \Phi\left(\dfrac{-i}{\sqrt{2}\sigma_{max}}\right)\dsum{i=1}{v^k}\dprod{j \neq i}{}\Phi_{i,j}.
\end{array}\setlength{\belowdisplayskip}{3pt}\end{equation*}

As $\Phi_{i,j}=1-\Phi_{j,i}$, this last expression is of the form studied in
Lemma~\ref{lemma:ProblemSumProd}. Thanks to this Lemma, one can propose the following upper bound:
\begin{equation*}\setlength{\abovedisplayskip}{3pt}
\Pr\left(x_*^k \notin L_i \mid \Cache^k\right) \leq \Phi\left(\dfrac{-i}{\sqrt{2}\sigma_{max}}\right)\times1
\underset{i~\textit{large enough}}{\leq} \exp\left(-i/\sqrt{2}\sigma_{max}\right).
\setlength{\belowdisplayskip}{3pt}\end{equation*}

Remind that $\kappa(i)$ denotes the index of the first iteration $k$ for which $x_*^k \notin L_i$. Recall also that $(A^i~\text{infinitely often}) \Leftrightarrow \forall i,~\kappa(i)<+\infty$. Then, $(A^i~\text{infinitely often})$ if and only if there is a sequence ($\kappa(i))_i \in \N^\N$ of iteration indexes satisfying $\forall i,~x_*^{\kappa(i)} \notin L_i$, knowing the set of evaluated points $\Cache^{\kappa(i)}$.

However, the sum $\dsum{i=0}{+\infty}\Pr\left(x_*^{\kappa(i)}\notin L_i \mid \Cache^{\kappa(i)}\right) \leq \dsum{i=0}{+\infty}\exp\left(-i/\sqrt{2}\sigma_{max}\right)$ converge. 
Then, the Borel-Cantelli Lemma ensures that each $(A^i)$ happens infinitely often with probability zero, because there is probability zero that the sequence of $(\kappa(i))_i$ which generates $(A^i)$ exists.
\end{proof}
%-----------------------------------%

Theorem~\ref{theorem:UpperBoundLevelSet} ensures that any execution of the algorithm generates a bounded sequence of incumbents. 
Therefore, the sequence of mesh size parameters necessarily approaches zero, in the following sense:
%-----------------------------------%
\begin{theorem}[The mesh becomes refined infinitely often]\label{theorem:MeshSize}
    Under the assumption that $f$ has all its level sets bounded, the inferior limit of the sequence $(\delta_m^k)_k$ is almost surely zero:
    \begin{equation*}
        \Pr\left(\underset{k\rightarrow+\infty}{\liminf}~\delta_m^k=0\right)=1.
    \end{equation*}
\end{theorem}
%-----------------------------------%

%-----------------------------------%
\begin{proof}
Recall the connection between $\delta_m^k$ and $\delta_p^k$: $\delta_m^k=\min(\delta_p^k,(\delta_p^k)^2)$. Recall also, from Theorem~\ref{theorem:UpperBoundLevelSet}, that there almost surely exists a constant $R$ such that $\forall k \in \N,~x_*^k \in B_R(x_*^0)$. If $\delta_m^k$ becomes too large, iteration $k$ necessarily generates candidates outside of $B_R(x_*^0)$:
$$ \delta_m^k > 2R \implies \M^k \cap B_R(x_*^0) = \{x_*^k\} \implies \P^k \cap B_R(x_*^0) = \emptyset.$$
Therefore, iteration $k$ is not a success (otherwise, it would have generated a new incumbent $x_*^{k+1}$ outside of $B_R(x_*^0)$). 
However, $\delta_m^{k+1} > \delta_m^k$ is impossible if iteration $k$ is not a success. 
This ensures that the sequence $(\delta_m^k)_k$ has an upper bound $\delta_m^{sup}$.

With this argument, one can deduce there is almost surely a ball ($B_{3R}(x_*^0)$) which contains all the points evaluated by the algorithm during the entire optimization process: $\forall k,~\forall x,~\Cache^k(x) \neq \emptyset \Rightarrow x \in B_{3R}(x_*^0)$.

Now assume, for the sake of contradiction, that there is a strictly positive minimal mesh size parameter, of index $k_{min}$: $\exists k_{min} \mid \forall k\in\N,~\delta_m^k\geq\delta_m^{k_{min}}=\delta_m^{min}$.
Lemma~\ref{lemma:MinimalMesh} gives 
    $\forall k,~\forall x,~\Cache^k(x) \neq \emptyset \Rightarrow x \in \M_{\delta_m^{min}}(x_*^0)$: the thinnest mesh contains all the evaluated points.
One also have $\delta_p^{min}$ such that $\delta_m^{min}=\min(\delta_p^{min},(\delta_m^{min})^2)$.

One can deduce that the algorithm evaluate a finite number of points, because the set of all points which could be visited is $B_{3R}(x_*^0) \cap \M_{\delta_m^{min}}(x_*^0)$, which is finite. Among this set, there is a subset $E$ of points visited infinitely often.

Assume also the minimiser of $f$ over $E$ is unique. As $E$ is finite,
Lemma~\ref{lemma:ConsistentMinimiser} is applicable, so there is almost surely an index $\kappa$ for which $\forall k\geq\kappa$, the estimates $f^k(\cdot)$ and truth objective function $f(\cdot)$ share the same minimiser over $E$. Any iteration $k \geq \kappa$ is necessarily not a success, therefore two situations might happen. If $p^k \in [0,0.5]$ becomes close enough to $0$ ($p^k < \beta_\ell$), then $\delta_p^{k+1} = \delta_p^k/2$. Otherwise, precision index $r^{k+1}$ becomes higher than $r^k$. As the sequence $(\delta_p^k)_k$ is assumed to be bounded by $\delta_p^{min}>0$, the precision index $r^k$ grows arbitrarily high (and then, standard deviation $\sigma^k(\cdot)$ of $x_*^k$ and other elements from $E$ becomes arbitrarily close to $0$). However, due to Lemma~\ref{lemma:PValueLimit}, this implies $p^k \underset{k\rightarrow+\infty}{\longrightarrow}0$. So there is an index $\ell$ such that $p^\ell<\beta_\ell$. This implies a contradiction: $\delta_p^{\ell+1}<\delta_p^\ell=\delta_p^{min}$.

If the minimiser of $f$ over $E$ is not unique, there is at least two points $x_1$ and $x_2$ for which $f(x_1)=f(x_2)$. The situation is analgous if there is more than two minimisers.
Denote $\kappa$ an index for which $\forall k \geq \kappa,~\forall x \neq x_1, x_2,~f^k(x) > f(x_1)=f(x_2)$. Denote also $\delta_m^{max}$ the largest mesh size such that $\{x_1,x_2\} \subset \M_{\delta_m^{max}}(x_1)$, which is necessarily smaller than $\delta_m^{sup}$. At iteration $k\geq\kappa$, $x_*^k$ can either be $x_1$ or $x_2$, depending on the values of $f^k(x_1)$ and $f^k(x_2)$. If $\delta_m^k$ becomes larger than $\delta_m^{max}$, iteration $k$ is necessarily not a success because the minimiser over $\M^k$ becomes unique. For $\delta_m^k \in \{\delta_m^{min},2\delta_m^{min},\dots,\delta_m^{max}\}$, the p-value $p^k$ is driven by a behaviour described in Lemma~\ref{lemma:PValueLimitPtsEgaux}. Then, $\delta_p^{k+1}=2\delta_p^k$ with probability $1-\beta_u$, $\delta_p^{k+1}=\delta_p^k$ with probability $\beta_u-\beta_\ell$, and $\delta_p^{k+1}=\delta_p^k/2$ with probability $\beta_\ell$. The sequence $(\delta_p^k)_k$ can therefore be seen as a stochastic process with the following behaviour:
$$\left\{\begin{array}{rcl}
    \delta_m^k=\delta_m^{max} & \Rightarrow & \delta_p^{k+1}=
        \left\{\begin{array}{ll}
            \delta_p^k   & \textit{with probability}~1-\beta_\ell, \\
            \delta_p^k/2 & \textit{with probability}~\beta_\ell,
        \end{array}\right. \\
    \\
    \delta_m^{min}\leq\delta_m^k<\delta_m^{max} & \Rightarrow & \delta_p^{k+1}=
        \left\{\begin{array}{ll}
            2\delta_p^k  & \textit{with probability}~1-\beta_u, \\
            \delta_p^k   & \textit{with probability}~\beta_u-\beta_\ell, \\
            \delta_p^k/2 & \textit{with probability}~\beta_\ell.
        \end{array}\right. \\
\end{array}
\right.$$
As such, results about stochastic processes ensures that $(\delta_p^k)_k$ reaches $\delta_p^{min}/2$ at least once, with probability one. This contradicts the definition of $\delta_p^{min}$.
\end{proof}
%-----------------------------------%

Theorem~\ref{theorem:MeshSize} ensures that there exists an infinite number of iteration indexes $K \subseteq \N$ such that $\underset{k \in K}{\lim}\delta_m^k=0$. Considering that the sequence $(x_*^k)_{k\in K}$ have all its elements included in a compact (the ball defined by
Theorem~\ref{theorem:UpperBoundLevelSet}), there exists another infinite set $L \subseteq K$ such that $(x_*^k)_{k\in L}$ converges. Then, following the definition of a \textit{refining sequence} given in~\cite{AuDe2006}, $(x_*^k)_{k \in L}$ is a refining sequence, and its limit $x_*$ is a refined point: the sequence $(x_*^k)_{k\in L}$ converges to $x_*$ and $(\delta_m^k)_{k\in L}$ converges to $0$.

Also, for any unitary direction $d$ there is a subsequence of indexes such that $d$ is a refining direction (as defined in~\cite{AuDe09a}) for that subsequence: $\exists (d^k)_{k\in L}$ a sequence of directions such that $\frac{d^k}{\norm{d^k}_2} \rightarrow d$ and $\forall k,~x_*^k+d^k \in \M^k$ and $\norm{d^k}_\infty\leq\delta_p^k$.

The following property is satisfied by $x_*$:
%-----------------------------------%
\begin{theorem}[Limit point given by the optimization process satisfies optimality conditions]\label{theorem:ClarkeOptimality}
    Recall it is assumed in this section that $D_f=\R^n$ and $f$ has its level sets bounded. The optimization process almost surely generates at least one refined point $x_*$ which satisfies:
    \begin{equation*}
        \Pr\left(\forall d~\textit{unitary},~f^\circ(x_*;d) \geq 0\right)=1
    \end{equation*}
    where $f^\circ(x;d)$ is the Clarke-derivative of $f$ at point $x$ in the direction $d$,
    defined in~\cite{Clar83a}.
\end{theorem}
%-----------------------------------%

%-----------------------------------%
\begin{proof}
Recall the definition of $f^\circ$ and Result 3.9 from~\cite{AuDe2006}:
$$ \begin{array}{rccl}
    f^\circ(x,d) & = & \underset{y \rightarrow x,~ t \searrow 0}{\limsup} & \dfrac{f(y+td)-f(y)}{t} \\
                 & = & \underset{y \rightarrow x,~ t \searrow 0,~v \rightarrow d}{\limsup} & \dfrac{f(y+tv)-f(y)}{t}.
\end{array} $$
Denoting $d_u^k=d^k/\norm{d^k}_2$ and $\delta^k=\norm{d^k}_\infty$, one can deduce
$$ \begin{array}{rccl}
    f^\circ(x_*,d) & \geq & \underset{k \rightarrow +\infty}{\limsup} & \dfrac{f(x_*^k+\delta^kd_u^k)-f(x_*^k)}{\delta^k} \\
     & = & \underset{k\rightarrow+\infty}{\lim}~\underset{\ell\geq k}{\sup} & \dfrac{f(x_*^\ell+\delta^\ell d_u^\ell)-f(x_*^\ell)}{\delta^\ell}.
\end{array} $$
Then, the proof is complete if the (random) refining sequence $(x_*^k)_{k\in L}$ satisfies
$$ \Pr\left(\underset{k\rightarrow+\infty}{\lim}~\underset{\ell\geq k}{\sup}~~ \dfrac{f(x_*^\ell+\delta^\ell d_u^\ell)-f(x_*^\ell)}{\delta^\ell}~~\geq0\right)=1 $$
or, equivalently
$$ \Pr\left(\underset{k\rightarrow+\infty}{\lim}~\underset{\ell\geq k}{\sup}~~f(x_*^\ell+\delta^\ell d_u^\ell)-f(x_*^\ell)~~<0\right)=0. $$
With an equivalent formulation of the $\limsup$, one can rewrite this as
$$ \Pr\left(\exists\mu>0,\exists k\in\N \textit{ such that } \forall\ell\geq k,~f(x_*^\ell+\delta^\ell d_u^\ell)-f(x_*^\ell)\leq-\mu\right)=0. $$
Recall that $\forall \ell,~x_*^\ell+\delta^\ell d_u^\ell \in \Cache^\ell$ and $x_*^\ell \in \argmin\{f^\ell(x),~x\in\Cache^\ell\}$. As a consequence, $x_*^\ell$ satisfies $f(x_*^\ell+\delta^\ell d_u^\ell)-f(x_*^\ell)\leq-\mu$ if the following event is realised:
$$ A_\mu^\ell: f^\ell(x_*^\ell)<f^\ell(x_*^\ell+\delta^\ell d_u^\ell),\textit{ knowing } f(x_*^\ell+\delta^\ell d_u^\ell)\leq f(x_*^\ell)-\mu.$$
However, recall also that $f^\ell(\cdot)\sim\Normale{f(\cdot)}{\sigma^\ell(\cdot)^2}$ and the $f^\ell(\cdot)$ are independent. So:
$$ \Pr(A_\mu^\ell)\leq\Pr\left(X<0 \mid X\sim\Normale{\mu}{2\sigma_{max}^2}\right)=\Phi\left(\dfrac{-\mu}{\sqrt{2}\sigma_{max}}\right)<1. $$
Then, $\forall \mu>0$, there is almost surely all but a finite number of $A_\mu^\ell$ which are realised. So the result holds because there is a probability zero that a threshold $\mu>0$ satisfies the condition.
\end{proof}
%-----------------------------------%

Theorem \ref{theorem:ClarkeOptimality} concludes the convergence analysis. As a summary, assuming that:
\begin{itemize}
    \item $f$ is defined on $\R^n$ entirely (although this could be restricted to $x_*$ relying in the interior of $D_f$), and all the level sets of $f$ are bounded,
    \item the lower bound $\sigma_{min}$ of the precision function $\rho$ is zero,
    \item the algorithmic parameters $\beta_\ell$ and $\sigma_{max}$ satisfies $0 < \beta_\ell \leq 0.5$ and $\sigma_{max} < +\infty$,
    \item the evolution of the precision index $r^k$ satisfies the $C_\textsc{Dp}$ condition in~\eqref{eq:minimalconditions},
\end{itemize}
any algorithm following the structure introduced in Algorithm~\ref{algo:dp/mp-mads} with the minimal requirement ($C_\textsc{Dp}$) on the $UpdateR$ function generates a refined point $x_*$ satisfying the Clarke necessary optimality conditions.

%===================================%
%===================================%
\section{Computational study}
\label{section:NumericalAnalysis}
%===================================%

This section compares an implementation of the algorithms \dpmads and \mpmads with the \nomad~\cite{Le09b} implementation of the \robustmads algorithm.
A first set of comparisons are presented on two analytical problems.
Then, Section~\ref{section:VME}
    compares the algorithms on a real and computationally expensive 
    Monte-Carlo stochastic problem.
\dpmads and \mpmads are implemented in Python 3.6, while \robustmads implementation is taken from \nomad~3.9.1. The first two tests, performed on analytical functions, are executed on an Intel® Core™ i5-8250U CPU @ 1.60GHz with 8 cores and 8GB of RAM.

%===================================%
\subsection{Comparison of stochastic algorithms}
\label{section:Profiles}
%===================================%

Comparisons of algorithms are commonly performed using some tools like the \textit{performance profiles}~\cite{DoMo02}, \textit{data profiles}~\cite{MoWi2009} and \textit{accuracy profiles}~\cite{BeHaLu17}. 
However, these profiles are inappropriate in an adaptive precision context. 
Usually, their discriminating criteria is the number of blackbox calls. This metric is irrelevant in adaptive precision context, as a few calls with great precision can be more expensive than many calls with low precision. 
One needs to adapt these profiles to a relevant metric: the computational effort. This is considered through th Monte-Carlo draws consumption.
Two situations can arise in any adaptive precision problems. 
If the noise magnitude is chosen directly by a number of Monte-Carlo draws, 
    then the metric is trivially set to that number. 
Otherwise, one may create a fictive Monte-Carlo simulation which gives equivalent results for a given number of draws. 
This exploits a well-known approximation of Monte-Carlo estimates : denoting $\widetilde{\mathcal{A}_N}$ an estimate of $\mathcal{A}$ coming from $N$ Monte-Carlo draws, there exists a constant $C$ such that $\widetilde{\mathcal{A}_N} \sim \Normale{\mathcal{A}}{C/N}$. 
Then, considering $C=1$ for simplicity, an estimate obtained with a standard deviation $\sigma$ can be interpreted as the result of a Monte-Carlo simulation with $N=1/\sigma^2$ draws. 
Thus, $1/\sigma^2$ is a metric which can be interpreted as a Monte-Carlo draws consumption. Former profiles are modified with this new metric. The fundamental object they all use is the accuracy of a given algorithm within a given budget : $$f^N_{acc}=\dfrac{f\left(x^N\right)-f\left(x_*^0\right)}{f\left(x_*\right)-f\left(x_*^0\right)},$$ where $x_*^0$ is the initial point, $x^N$ the incumbent found with a budget of $N$ Monte-Carlo draws, $x_*$ is the optimum of $f$, and $f(\cdot)$ is the true value of points (if known) or its estimated value $f^k(\cdot)$ otherwise. It is therefore possible to determine the minimal budget required by an algorithm $a$ to solve a problem $p$ with a given tolerance $\tau$. The following formula gives this budget: $N_{a,p}\in\underset{N\in\R^+}{\argmin}\left\{f^N_{acc} \mid f^N_{acc}\geq1-\tau\right\}$ if such $N$ exists, and $+\infty$ otherwise. Although it is not used in the following graphs, an alternative formula giving the decimal logarithm of this budget could be considered: $N_{a,p}^{log}\in\underset{N\in\R^+}{\argmin}\left\{f^{10^N}_{acc} \mid f^{10^N}_{acc}\geq1-\tau\right\}$.

With these quantities, \textit{performance} and \textit{accuracy profiles} can be constructed in a way similar to their deterministic equivalent. However, \textit{data profiles} have to be more deeply modified. With the original profiles, the abscissa represents a number of calls divided the number of variables ($\frac{k}{n+1}$). As a positive basis of $\R^n$ requires at least $(n+1)$ vectors, $\frac{k}{n+1}$ represents the number of positive bases that could have been created within a budget of $k$ blackbox calls. As this is no longer relevant, the profiles are modified. Remind that to guarantees a given standard deviation $\sigma$ to the output of a blackbox call, $N=1/\sigma^2$ draws are required. The modified data profiles, defined for a reference standard deviation, represents in abscissa the quantity $\frac{k}{N}$, the number of estimates at guaranteed standard deviation $\sigma$ which could have been computed within a budget of $k$ draws.

%===================================%
\subsection{Analytical problems}
\label{section:AnalyticalProblems}
%===================================%

The first analytical problem discussed here, named Norm2, is easy to solve in the deterministic context. It is used to compare algorithms during the intensification close to an optimum.
\begin{equation*}
    \underset{(x,y)\in \R^2}{\min} \ \norm{(x,y)}_2
    \mbox{ with } \left(x_*^0,y_*^0\right)  =  \left(\pi^2,e^2\right). \\
\end{equation*}

Its noisy equivalent applies a noise $\Normale{0}{\sigma^2}$ at any computation of $\norm{(x,y)}_2$, with $\sigma$ decided by algorithms. The equivalent number of Monte-Carlo draws is $N=1/\sigma^2$. The stopping criteria is related to the frame parameter: $\delta_p<10^{-10}$. The noisy problem is:
\begin{equation}
    \underset{(x,y)\in \R^2}{\min} \ \left( \underset{\#\Cache^k(x,y)\rightarrow +\infty}{\lim}f^k\left(x,y\right) \right)
    \mbox{ with }
    \left\{\begin{array}{l}
                \left(x_*^0,y_*^0,\delta_p^{min}\right)  =  \left(\pi^2,e^2,10^{-10}\right), \\
                f^k\text{ and }\sigma^k \text{ defined as in Section~\ref{section:Notations}} \\
                \text{via }
                f_\sigma(\cdot) = \norm{\cdot}_2+\Normale{0}{\sigma^2}.
            \end{array}\right.
\end{equation}

Figure~\ref{fig:Norm2-Results-CVMC} shows convergence versus consumption. 
\robustmads is used with the standard deviation fixed to $\sigma=10^{-10}$. Preliminary tests shows that the algorithm struggle to reach an objective value lower than $\sigma/10$, which is lower than the chosen stopping criteria. 
One can observe that all \robustmads runs are close: it always reaches an objective function value of $10^{-10}$ after approximately $3.8\times10^{22}$ draws and cannot intensify more because $\sigma$ becomes high compared to the small variations of $f$ around the optimum. Meanwhile, \dpmads and \mpmads successfully goes closer to the optimum. However, \dpmads seems more reliable than \mpmads: at a given budget it reachs a lower objective. Also, all its runs converges within an equivalent budget ($10^{21}$ to $10^{23}$ draws) while some of \mpmads runs requires up to $10^{28}$ draws.

\begin{figure}[!ht]
    \centering
    \subfigure{\includegraphics[width=0.45\textwidth]{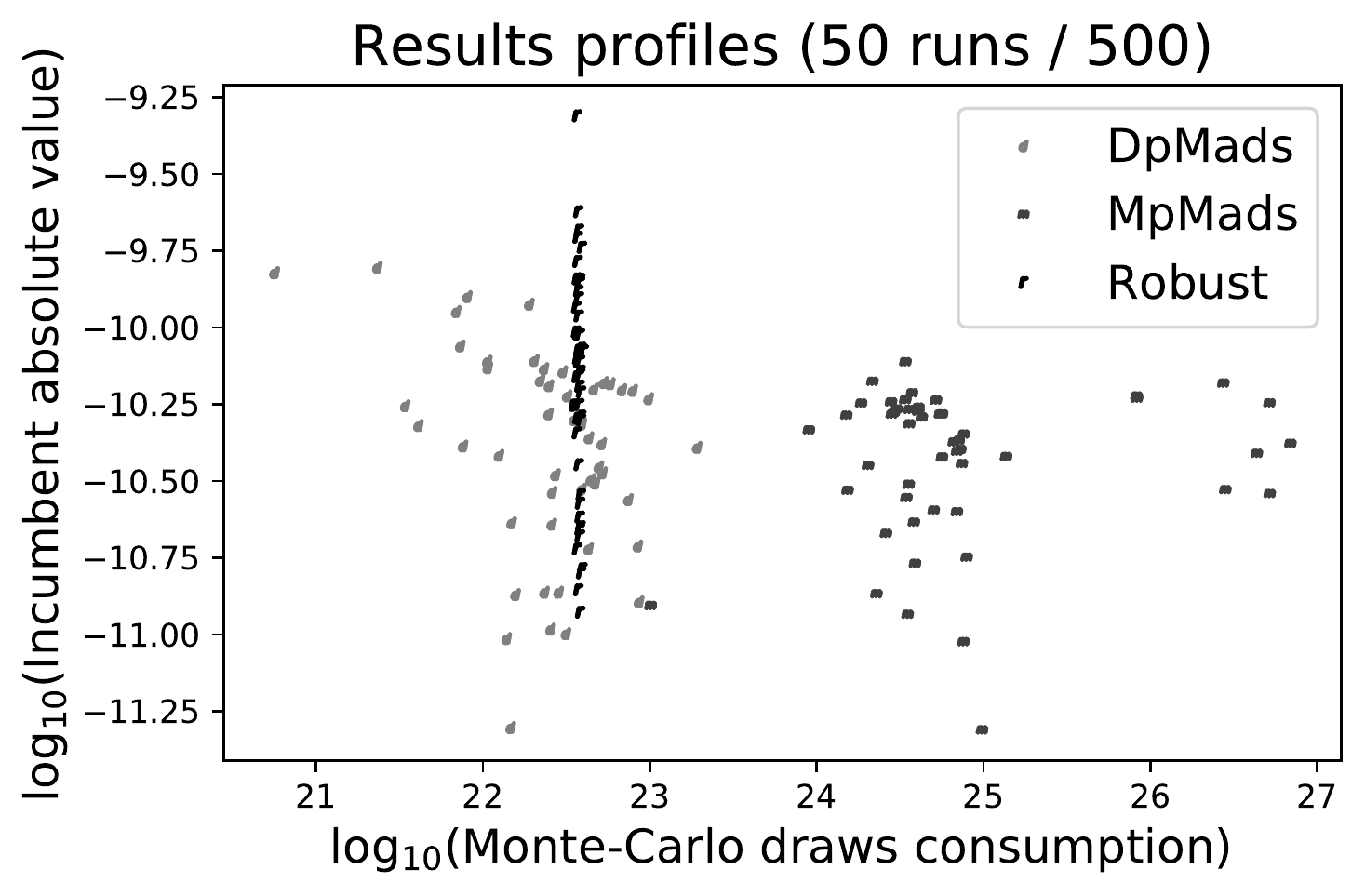}}
    \subfigure{\includegraphics[width=0.45\textwidth]{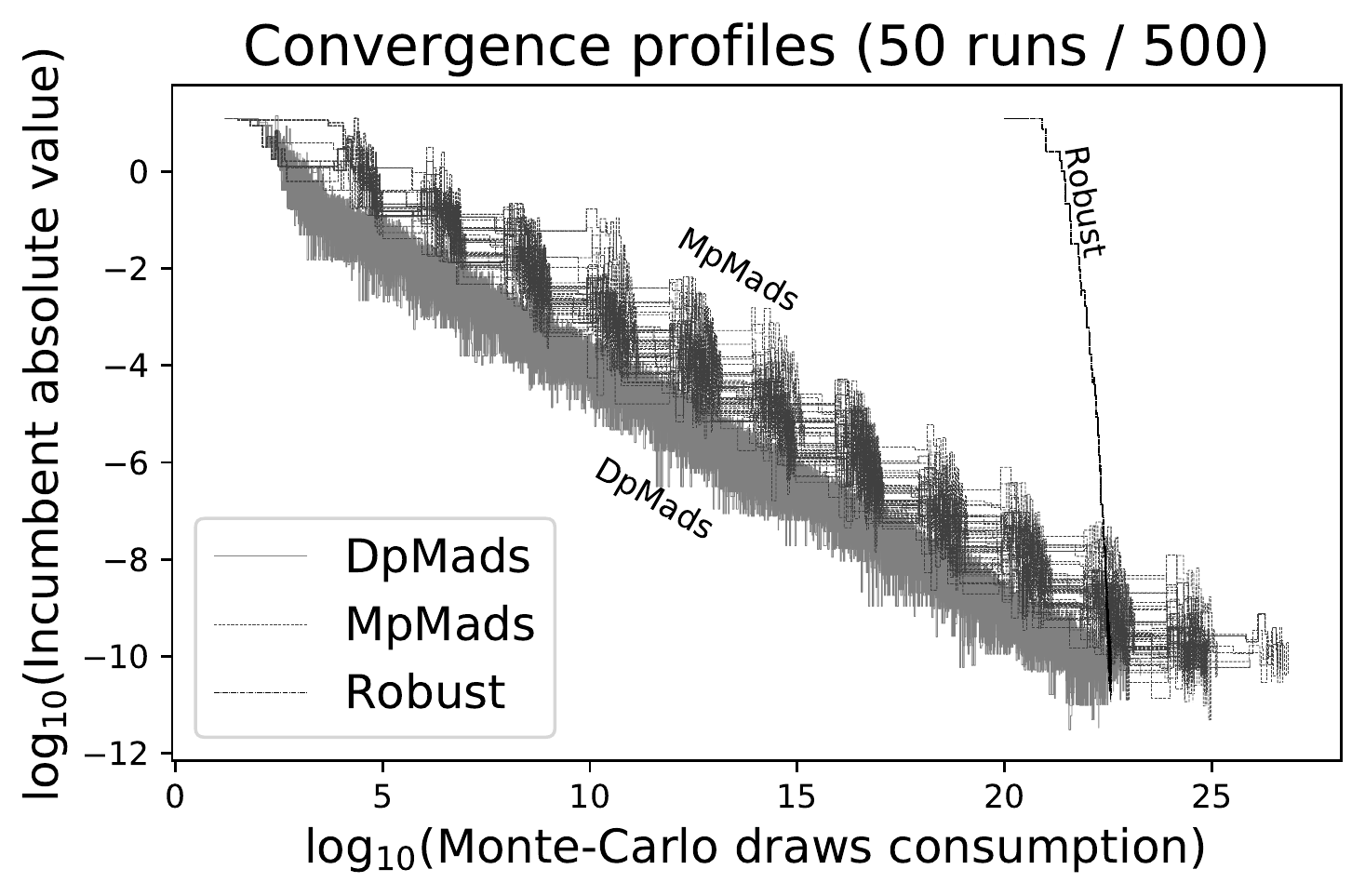}}
    \caption{\guillemets{Norm2} - (a) Results, (b) Monte-Carlo convergence profiles.}
    \label{fig:Norm2-Results-CVMC}
\end{figure}

This is also shown by the profiles in Figure~\ref{fig:Norm2-Perf-Data-Qual}. 
The accuracy profiles show that $10^{23}$ draws is the minimal budget required to make all the algorithms to converge at good precision (while a lower budget makes \robustmads to fail and \mpmads to be dominated by \dpmads). 
With the performance and data profiles, it appears that for any tolerance greater than $10^{-13}$, \dpmads outperforms the other two algorithms, notably around $\tau=10^{-10}$ or $10^{-11}$. 
The precision reference for the data profiles is $10^6$ draws ($\sigma=10^{-3}$).

\begin{figure}[!ht]
    \centering
    \subfigure{\includegraphics[width=\textwidth]{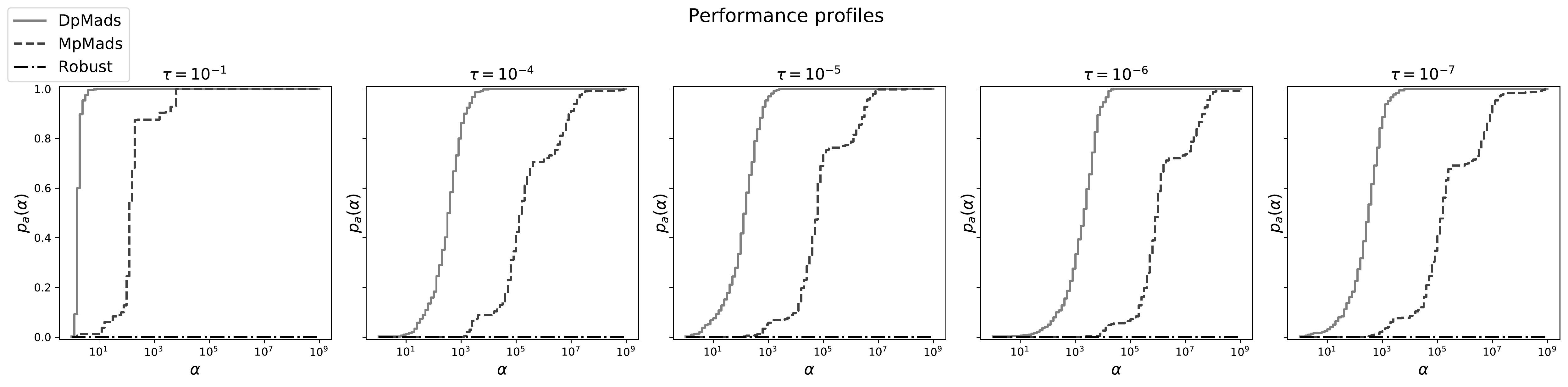}}
    \subfigure{\includegraphics[width=\textwidth]{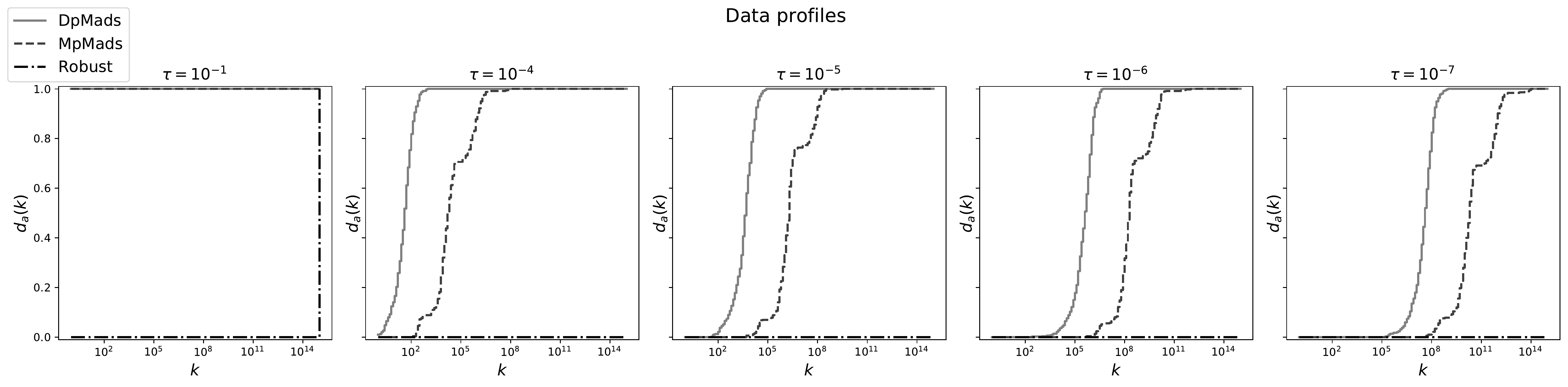}}
    \subfigure{\includegraphics[width=\textwidth]{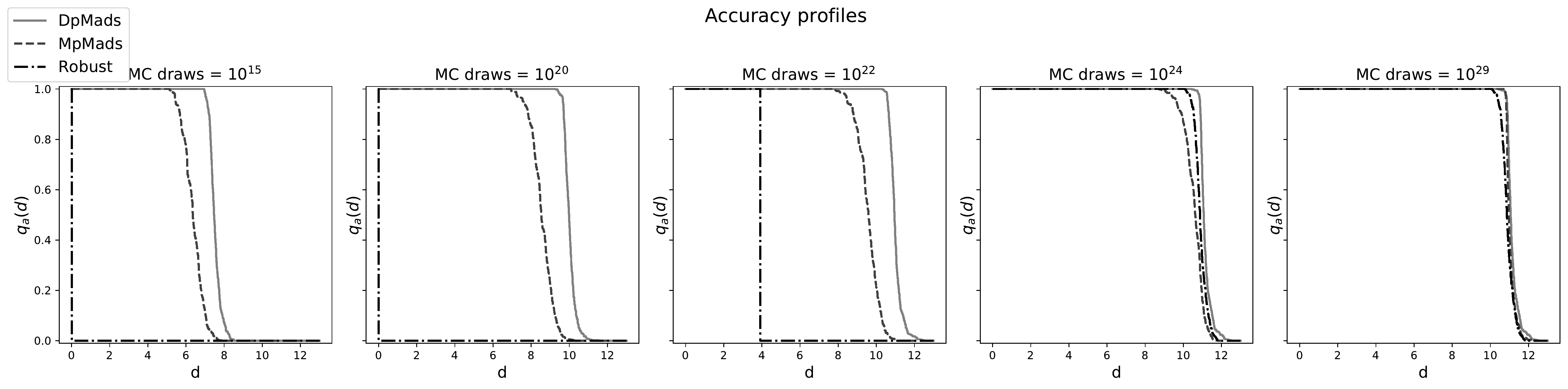}}
    \caption{\guillemets{Norm2} - (a) Performance profiles, (b) Data profiles, (c) Accuracy profiles.}
    \label{fig:Norm2-Perf-Data-Qual}
\end{figure}

The second analytical problem, denoted \guillemets{Moustache}, aims to compare the algorithms during their exploration process, in a restrictive space of variables.
The domain is the thin region illustrated in Figure~\ref{fig:GraphMoustache}.

\begin{figure}[!ht]
\centering
\includegraphics[width=0.75\linewidth]{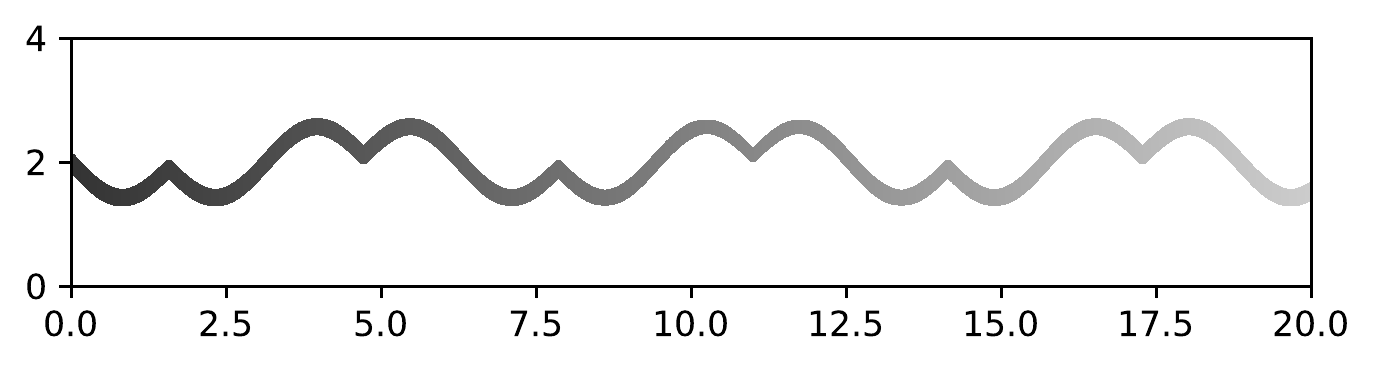}
\caption{The domain of the \guillemets{Moustache} function.}
\label{fig:GraphMoustache}
\end{figure}

Starting from the feasible point $(0,2)$, the objective is to maximise $x$ (thus minimise $f(x,y)=-x$), 
with $f$ not defined outside of the tight ribbon. At a given $x$, the interval of the admissible $y$ values is denoted $I(x)$, defined as:
    $$ \begin{array}{rcl}
            I(x) & = & [ g(x)-\varepsilon(x), g(x)+\varepsilon(x) ], \\
            g(x) & = & -(\abs{\cos(x)}+0.1)\sin(x)+2, \\
            \varepsilon(x) & = & l_{min}+(l_{max}-l_{min})\left(1-\dfrac{1}{1+\abs{x-x_m}}\right), \\
            (l_{min},l_{max},x_m) & = & (0.05,0.1,11).
    \end{array} $$
The noise appears at the computation of $-x$ in the objective. This does not mean that the variable $x$ itself is noisy, the noise appears when the value $-x$ is returned by the objective function. The equivalent Monte-Carlo consumption is $N=1/\sigma^2$. The problem is:
\begin{equation}
    \begin{array}{rl}
    \underset{(x,y)\in \R^2}{\min} & \left( \underset{\#\Cache^k(x,y)\rightarrow +\infty}{\lim}f^k\left(x,y\right) \right) \\
    \text{s.t.} & (x,y) \in [0,20] \times I(x)
    \end{array}
    \text{ with } \left\{\begin{array}{l}
                    \left(x_*^0,y_*^0,\delta_p^{min}\right) = \left(0,2,10^{-5}\right), \\
                    f^k\text{ and }\sigma^k \text{ defined as in Section~\ref{section:Notations} via} \\
                    f_\sigma(x,y) = -x+\Normale{0}{\sigma^2}.
                \end{array}\right.
\end{equation}

\robustmads is run with $\sigma=0.001$, a trade-off between quality of results (higher $\sigma$ fails more often) and consumption ($\sigma \leq 0.0005$ consumes at least $10^{10}$ Monte-Carlo draws but overall the results are equivalent). 
The problem is solved by \dpmads every time. 
As Figure~\ref{fig:Moustache-Results-CVMC} shows, the optimal value $-20$ is always reached, within a budget of up to $10^{11}$ Monte-Carlo draws. 
\dpmads appears to be robust to random effects: a budget of $10^7$ draws is actually always sufficient to solve the problem, the remaining is used trying to intensify around the frontier $x=20$. \mpmads consumes more (from $10$ to $10^3$ times more computational efforts) and has fewer guarantees of quality: it fails once to reach $-20$.

\begin{figure}[!ht]
    \centering
    \subfigure{\includegraphics[width=0.45\textwidth]{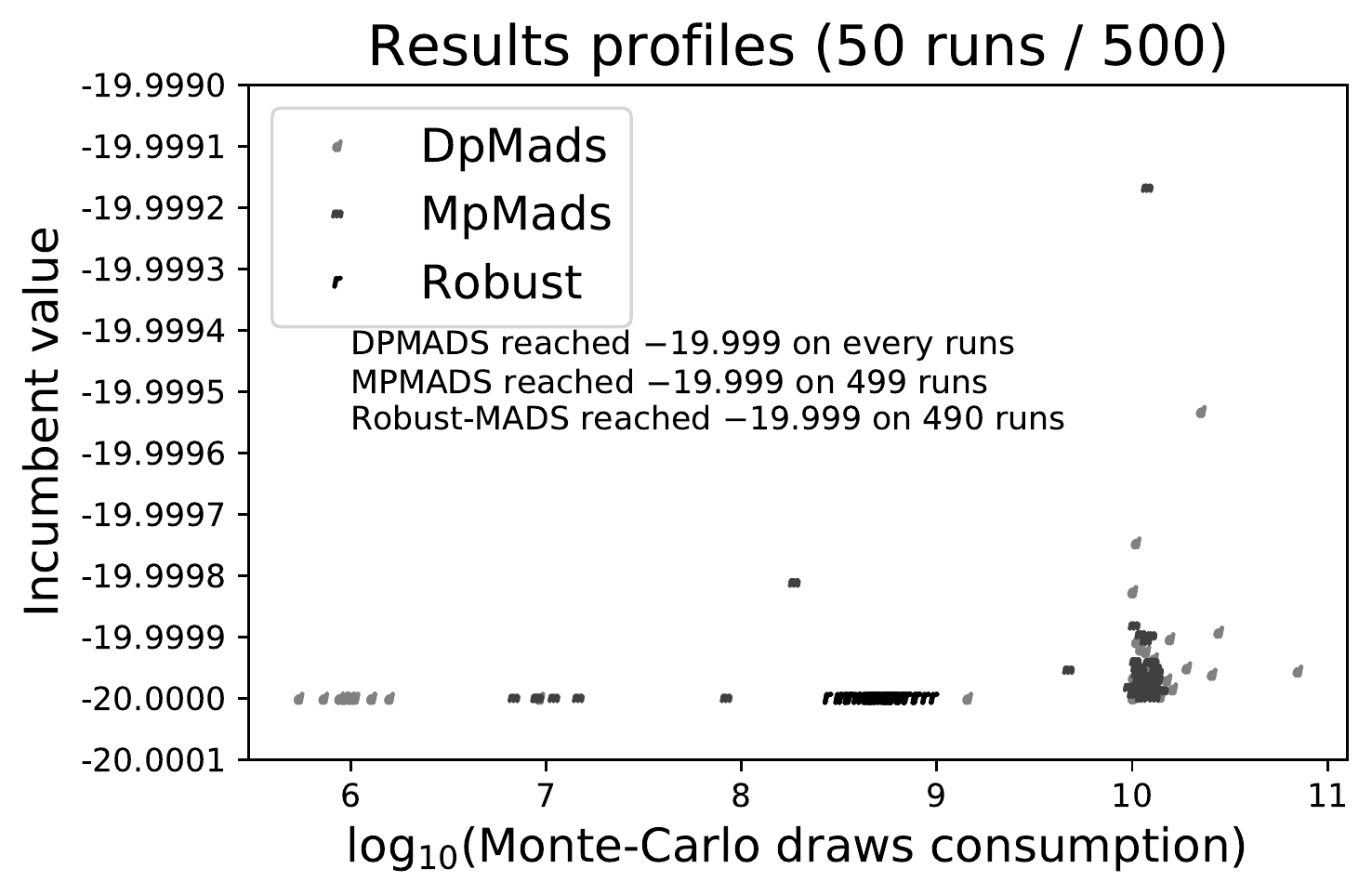}}
    \subfigure{\includegraphics[width=0.45\textwidth]{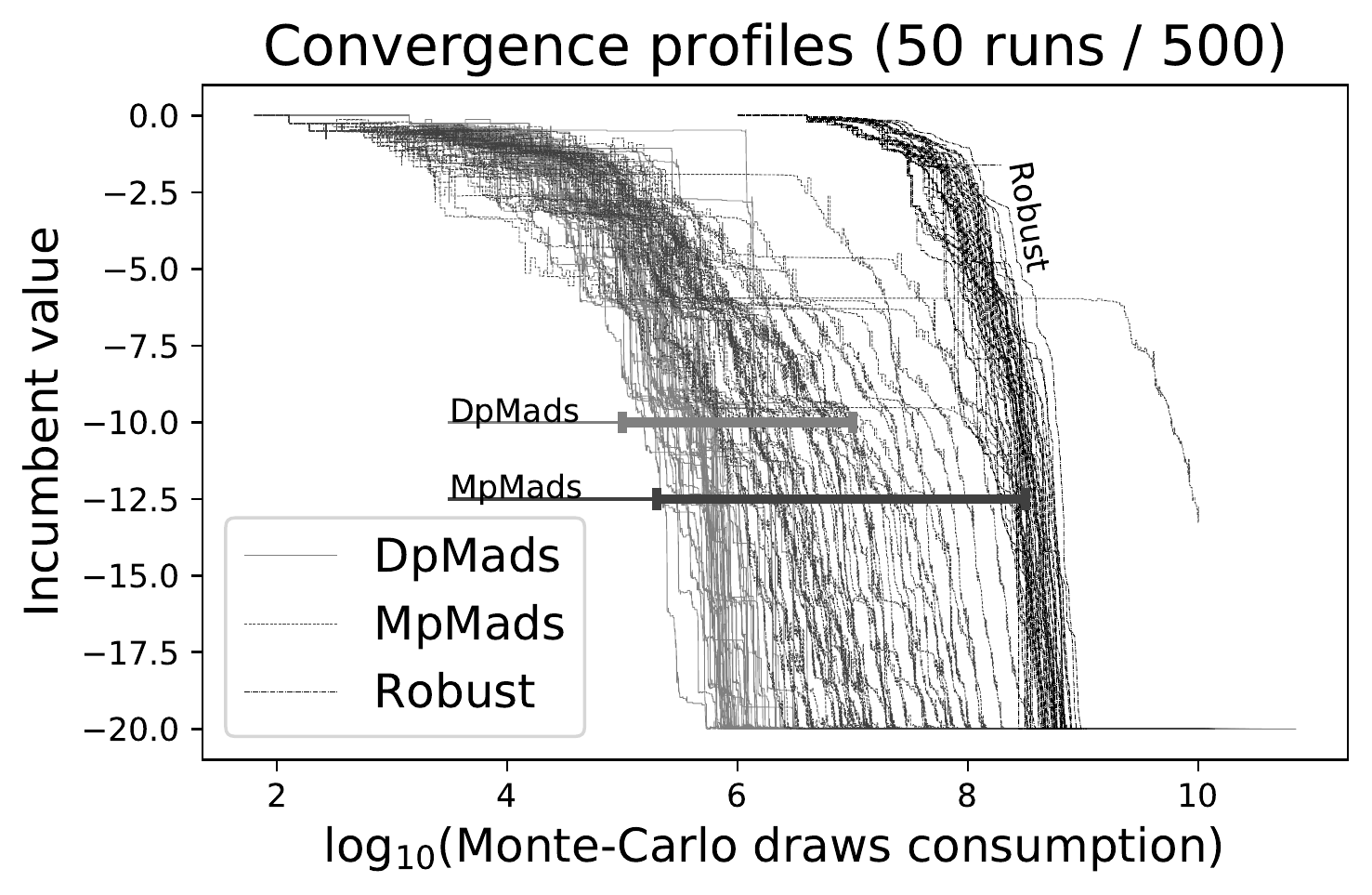}}
    \caption{\guillemets{Moustache} - (a) Results, (b) Monte-Carlo convergence profiles.}
    \label{fig:Moustache-Results-CVMC}
\end{figure}

This analysis can be recovered from the profiles in Figure~\ref{fig:Moustache-Perf-Data-Qual}. On the performance profiles, one can observe that \dpmads and \mpmads solves the problem every time with a tolerance $\tau\geq10^{-6}$ and starts to fail at $\tau=10^{-7}$. 
However, \mpmads struggles to reach the optimum as fast as \dpmads. 
The data profiles with reference precision of $10^6$ draws ($\sigma=10^{-3}$) show that \dpmads requires less effort to reach the optimum at a given tolerance than \mpmads. 
All \robustmads runs require an equivalent computational effort regardless of the tolerance. 
When it reachs the optimum, it performs well on accuracy profiles and attains the optimum at very low tolerance.
Because of the fact that the optimum is at the frontier of $D_f$, the algorithms generates numerous points at $x$ close to $20$. Then, the smoothing effect helps \robustmads to improve its incumbent. The other two algorithms do not have any smoothing effect, then they struggle to intensify as much as \robustmads.

\begin{figure}[!ht]
    \centering
    \subfigure{\includegraphics[width=\textwidth]{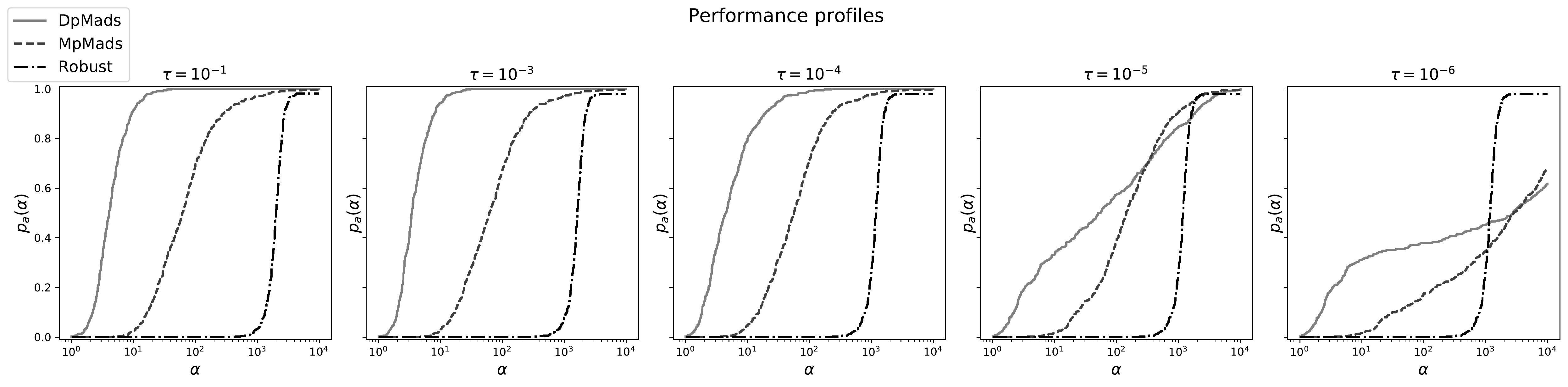}}
    \subfigure{\includegraphics[width=\textwidth]{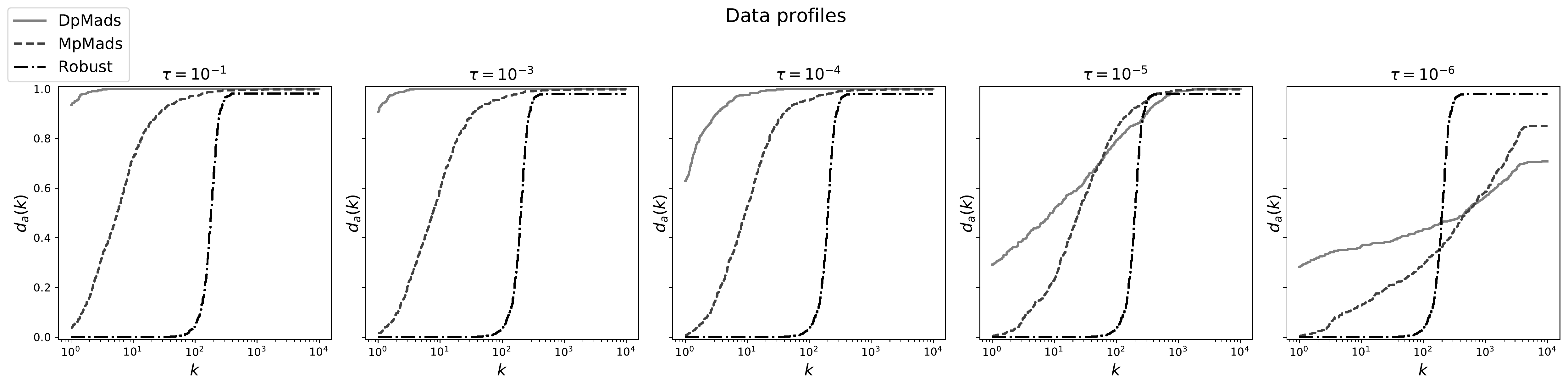}}
    \subfigure{\includegraphics[width=\textwidth]{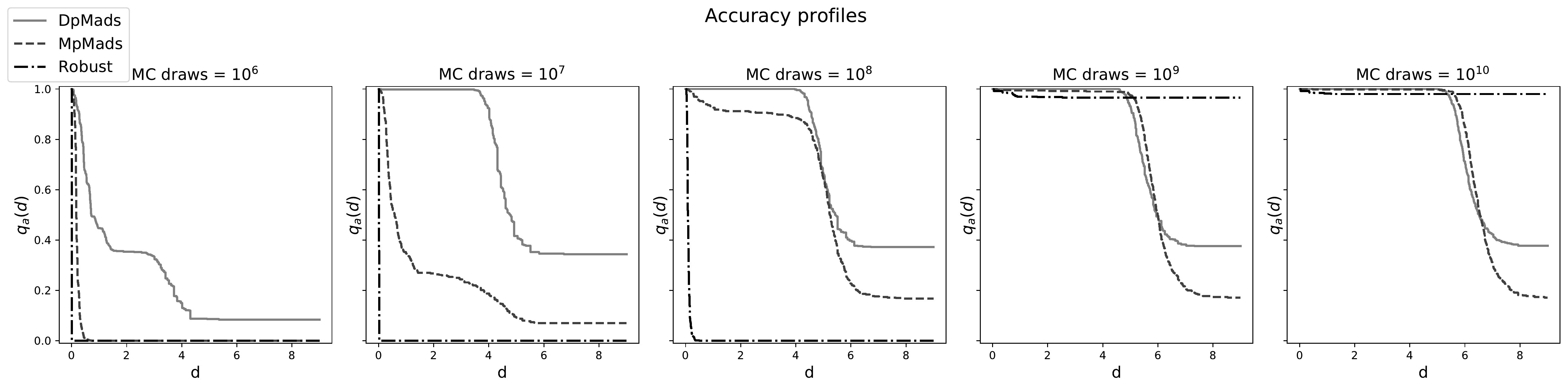}}
    \caption{\guillemets{Moustache} - (a) Performance profiles, (b) Data profiles, (c) Accuracy profiles.}
    \label{fig:Moustache-Perf-Data-Qual}
\end{figure}

Eventually, the \dpmads and \mpmads accuracy profiles decreases at $d=5$, because reaching higher values of $d$ means reaching $20$ at a distance smaller than $10^{-5}$ (the stopping criteria on $\delta_p$). 
This is made difficult because at such low distance between two points, the values of their images are close, so it becomes hard to make estimates sufficiently accurate.

%\clearpage
%===================================%
\subsection{Asset Management Problem}
\label{section:VME}
%===================================%

The two algorithms proposed in the present work are now compared with \robustmads on an asset management problem from~\cite{VME2016:BrIoGrLoRe}.
The instance considered has four assets relying on the same spare stock in case of failure.
The replacement date needs to be determined for each asset,
    and the acquisition date needs to be identified to replenish the stock.
The five dates must be from a 350-month horizon,
    which corresponds to a little more than 29 years,
    to maximise the expected Net Present Value (NPV).
Given a set of fixed dates,
    the NPV is computed using a tool called VME,
    developed by EDF~R\&D to evaluate asset investments~\cite{VME2017:Lonchampt}.
VME is a discrete-event simulator
    where an asset management strategy is tested against asset failures 
    that are randomly generated by Monte-Carlo methods.

The problem stated above is much more difficult than the one in \cite{VME2016:BrIoGrLoRe}
    where the variables were annual instead of monthly,
    and covered a 10-year horizon.

Let $x = (a_1, a_2, a_3, a_4, s)$ be a solution of the above problem
    where $a_i$ is the replacement date for the asset $i \in \{1,2,3,4\}$
    and $s$ the date at which a new spare is added to the stock.
Each of the five variables are integer and take any value from 1 to 350.

Computational experiments are performed on a HP~Z420 Workstation,
    Intel Dual-Core Xeon E5-1620~@~3.60GHz,
    RAM~32.0 Go, 64~bits,
    Windows~7~Pro~SP1.
The initial solution for all experiments is $x_*^0 = (240, 240, 132, 240, 120)$.

%===================================%
\subsubsection{Preconditioning}
%===================================%

In VME, the source of stochasticity lies a complex Monte-Carlo simulation over a so-called \guillemets{tree of scenarios}. 
Since the true value of the objective funciton is defined as the expected negative cashflow induced by a given choice of variables, knowing any potential scenario and its probability, it is difficult to recover the exact law followed by the noise. Therefore, some proactive choices have been made while analysing the problem.

Intense tests showed that the following law makes an acceptable approximation of the noise law.
It overestimates its magnitude on some points $x$ but never underestimates it.

The blackbox receives a value $\sigma$ from the algorithms and translates it to a number $N$ of Monte-Carlo draws. The tests showed that $N=2^{10}$ leads to a standard deviation $\sigma=1800$, and $\sigma$ is divided by $2$ when $N$ is multiplied by $4$. Denoting $f_N(\cdot)$ the Monte-Carlo approximation run by the blackbox with $N$ draws, the law of the noise is approximated by:
$$ f_\sigma(x) \sim \Normale{f(x)}{\sigma^2}= f_N(x)~\text{ with }~N=2^{10-2\log_2(\sigma/1800)}=\dfrac{2^{10}\times1800^2}{\sigma^2}. $$
Then, when the optimization algorithms runs with a given value of $\sigma$,
 the blackbox computes the corresponding number $N$ of draws, 
 then performs a Monte-Carlo simulation with these $N$ draws 
 and returns the output to the algorithm. 
This value can be interpreted by the algorithm as an observation of 
    $\Normale{f(x)}{\sigma^2}$, as required.
The least value of $N$ is set to $1024$ and therefore $\sigma_{max}=1800$.

During these tests, an hidden constraint may be triggered. 
The program restrict the number of Monte-Carlo draws to be at most 
    $2850812 \simeq 2^{21.4429}$, 
    and does not run any attempt to use more than this number. 
As a consequence, the standard deviation of the noise cannot fall arbitrarily close to zero: the smallest possible value of $\sigma$ is $\sigma_{min}=34.1144$. 
Thus, using monotonic strategies such ad \mpmads, the user needs to ensure that the precision is unlikely to grow too high, otherwise the algorithm may eventually ask for a $\sigma$ which cannot be computed by the blackbox. 
This restriction contributes to make the dynamic strategy interesting, as \dpmads naturally avoids to increase the precision as much as possible. 
To avoid problems, the value $\sigma_{min}=50$ is chosen on the tests.

The following tests compare \dpmads, \mpmads and \robustmads, using this preconditioned formula, and an implementation of the deterministic \mads algorithm. For \robustmads and \mads, the fact that the variables are all integer is integrated through the \textit{Granular Mesh} proposed in~\cite{G-2018-16}. For \dpmads and \mpmads, a lower bound on the $\delta_m$ mesh size is implemented and set to $1$, such that $(\delta_m^k)_k$ remains above this bound: $\delta_m^{k+1} \gets \delta_m^k/2$ is not computed if $\delta_m^k=1$.

%===================================%
\subsubsection{Results on the asset management problem}
%===================================%
\renewcommand{\tabcolsep}{4pt}

For the deterministic algorithm, a fixed number of draws per evaluation is fixed a priori and the algorithm runs as if the computation of $f_N$ were exact.
Table~\ref{tab:vanillaMADS} gives results using the \nomad~\cite{Le09b} implementation of \mads,
version~3.9.1, with neither models nor anisotropic mesh, and the speculative search turned off.
For comparisons, replications of runs were made on the instances that required less than two weeks to complete.
One can observe the influence of the noise: 
 with few draws per evaluation, objective function values
 $f^k(x_*)$ are high (around $8700$) while the runs with more draws propose reach objective values near $5400$. 
As this is a maximisation problem, optima on the first runs are overestimated by approximately $3000$ units.

\begin{table}[ht]
\center
\begin{tabular}{|r|c c c c c|c|c|c|r|r|}
\hline
draws / eval & \multicolumn{5}{c|}{$x_*$} & $f^k(x_*)$ & $\sigma^k(x_*)$ &Evals & Time (s) & Time  \\
\hline \hline
    1024 & 261 & 288 & 120 & 250 & 107 &  8668.33 &            1800 & 1889 &     1891  &   32 m \\
    1024 & 262 & 292 & 129 & 240 & 112 &  9005.12 &            1800 & 2050 &     2055  &   34 m \\ 
    1024 & 247 & 345 & 144 & 289 & 122 &  8371.24 &            1800 & 1891 &     1910  &   32 m \\ \hline
   10000 & 272 & 336 & 121 & 248 & 111 &  6140.15 &             576 & 2594 &    17636  &  5.9 h \\
   10000 & 247 & 332 &  84 & 209 &  66 &  5893.69 &             576 & 2247 &    15724  &  4.4 h \\
   10000 & 257 & 286 & 117 & 301 &  97 &  6209.33 &             576 & 3319 &    22419  &  6.2 h \\ \hline
  100000 & 267 & 281 & 119 & 229 & 108 &  5623.74 &             182 & 3268 &   214310  &  2.5 d \\
  100000 & 259 & 297 & 121 & 245 & 108 &  5630.17 &             182 & 2316 &   148752  &  1.7 d \\
  100000 & 260 & 304 & 125 & 248 & 112 &  5622.03 &             182 & 3582 &   229966  &  2.7 d \\ \hline
  500000 & 251 & 296 & 120 & 224 & 105 &  5431.04 &            81.5 & 2985 &   961232  & 11.1 d \\
  500000 & 268 & 243 & 133 & 213 & 119 &  5470.23 &            81.5 & 2324 &   736527  &  8.5 d \\ \hline
 1000000 & 257 & 292 & 130 & 226 & 118 &  5437.37 &            57.6 & 3403 &  2187009  & 25.3 d \\ \hline
\end{tabular}
\caption{\mads (\nomad~3.9.1).}
\label{tab:vanillaMADS}
\end{table}

\robustmads also requires a fixed number of draws per evaluations. 
Table~\ref{tab:robustMADS} reports results obtained with the \nomad implementation of \robustmads
    without the anisotropic mesh and with no speculative search. 
The column labelled $f_\sigma(x;\omega)$ represents the value computed by the blackbox on $x$, and $f^k(x)$ the estimated smoothed value proposed by \robustmads. 
Due to the \robustmads mechanics, the cache $\Cache^k(x)$ is either empty or contains a single observation. 
However, the precision $\sigma^k(x_*)$ is intractable because of the smoothing. 
One can observe that the smoothed values are coherent, because regardless of the number of draws, the proposed smoothed values are all close to $5500$.

\begin{table}[ht]
\center
\begin{tabular}{|r|c c c c c|c|c|c|r|r|}
\hline
draws / eval & \multicolumn{5}{c|}{$x_*$}  & $f_\sigma(x_*;\omega)$ & $f^k(x_*)$ & Evals & Time (s) &  Time  \\
\hline \hline
    100  & 268 & 306 & 145 & 191 & 132 & 11313.60 &   5601.01  & 36526 &   14557  &  4.0 h \\
    100  & 237 & 229 & 115 & 266 & 103 & 15456.60 &   5576.94  & 29707 &   12978  &  3.6 h \\
    100  & 300 & 321 & 145 & 245 & 134 & 11508.50 &   5803.96  & 36678 &   16728  &  4.7 h \\ \hline
   1024  & 257 & 273 & 132 & 208 & 119 &  7850.24 &   5434.91  & 21390 &   21507  &  5.9 h \\ 
   1024  & 267 & 247 & 124 & 295 & 112 &  9024.06 &   5276.66  & 22056 &   22338  &  6.2 h \\ \hline
  10000  & 261 & 283 & 133 & 221 & 119 &  6138.82 &   5367.90  & 22168 &  152875  &  1.8 d \\ \hline
 100000  & 275 & 261 & 134 & 230 & 121 &  5589.82 &   5352.53  & 21847 & 1408037  & 16.3 d \\ \hline
\end{tabular}
\caption{Robust-MADS (\nomad~3.9.1).}
\label{tab:robustMADS}
\end{table}

These results shows that if the number of draws per evaluation is fixed a priori, it needs to be high. With the accuracy gain provided by \robustmads' smoothing effect, this number can remain close to $10^5$ but then, the computation time is important ($16.3$ days). 
These observations, combined with the estimated optimal solution, $x_* = (275, 261, 134, 230, 121)$ with $f(x_*) \simeq 5352$, are used for comparisons with \dpmads and \mpmads.

For these two algorithms, preliminary tests showed that the $UpdateR$ function has a considerable influence. 
Since the noise magnitude is initially very high, the function need to allow a fast improvement of the precision index. 
Otherwise (as it appeared on the runs with such a situation), 
    numerous iterations are performed at very low precision and then, 
    the cache $\Cache^k$ becomes large and full of imprecise estimates which all are re-estimated by the $Search$.
This process consumes a noticeable amount of draws. 
The following results are generated by \dpmads with default parameters 
    except that the threshold $\tau$ in the $Search$ is increased to $40\%$, 
    and by \mpmads with an extended $UpdateR$ function allowing the precision increase to exceed one unit.
Figure~\ref{fig:VMEgrapheCV} illustrates one run for the three non-deterministic algorithms.
The shaded area around the curves depicts the estimated standard deviation of the incumbent values.
All the other runs have similar results.

\begin{figure}[!ht]
\centering
\includegraphics[width=0.75\linewidth]{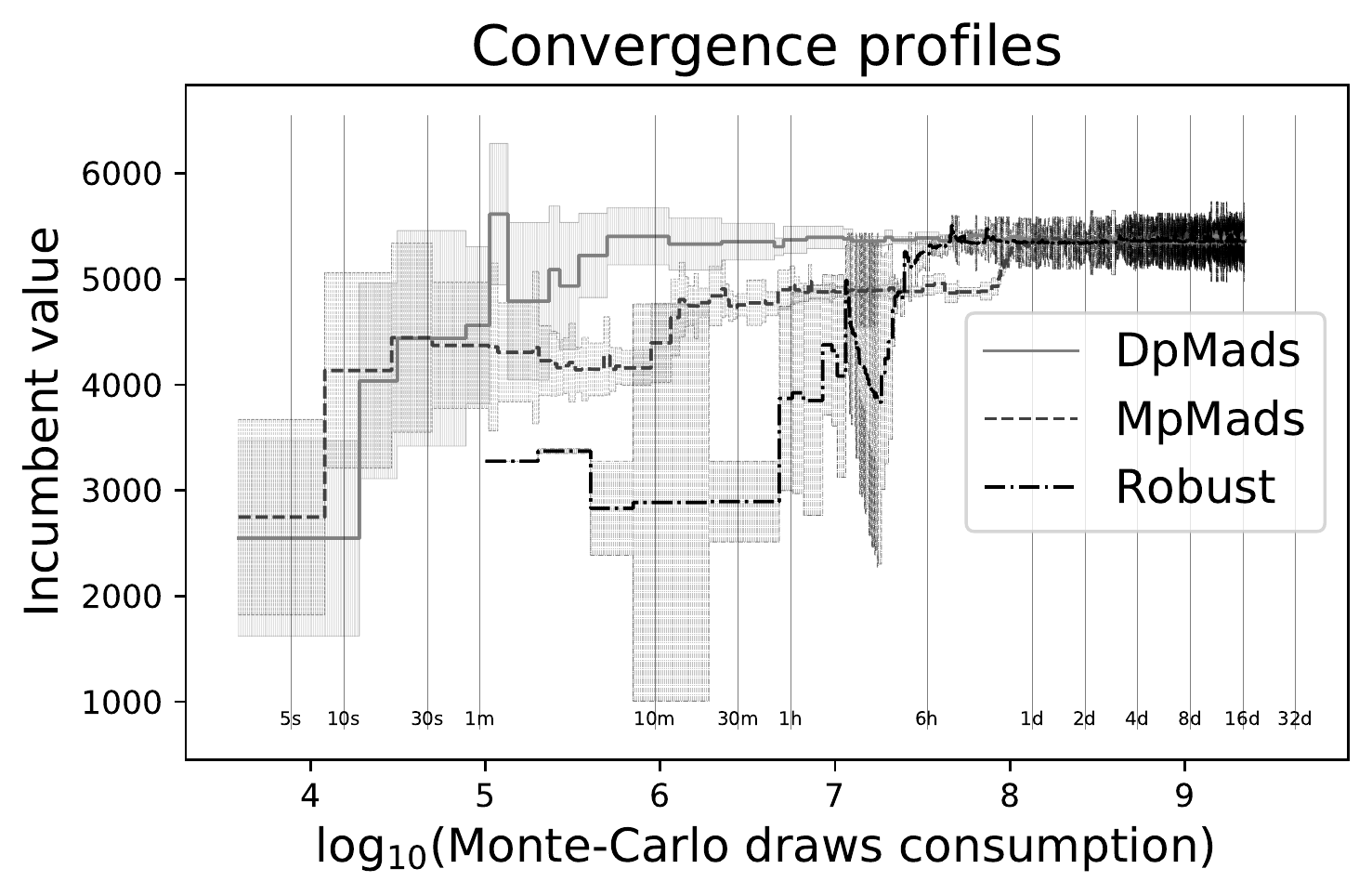}
\caption{Convergence graph, with estimated value of the incumbent in ordinate. The length of the vertical range around a curve is twice the estimated standard deviation of the incumbent value.}
\label{fig:VMEgrapheCV}
\end{figure}

The figure shows that \dpmads reaches a nearly constant incumbent function value after $10^6$ draws, while \robustmads stabilises only after $10^8$ draws. 
Also, the standard deviation of \dpmads' incumbent is close to $100$ while \dpmads stabilises around it.
All the computational efforts performed after are dedicated to the reduction of the standard deviation, with no change of incumbent. After $10^{7}$ draws, because of the search and the precision index going high, the incumbent and other quasi-optimal solutions have a standard deviation $\sigma^k < 10$. 
In summary, 
    to reach the optimal solution,
    \robustmads requires $100$ times more draws than \dpmads.
\dpmads uses the remaining budget to intensively analyse the most promising solutions.
A detailed analysis of the precision index $r$ shows that it starts to improves strongly at that point of the optimization. Therefore, the $Poll$ step (Algorithm~\ref{algo:poll}) does not contributes anymore to the optimization, because all the candidates already satisfies $\sigma^k(x) < \sigma^k = \rho(r^k) \simeq \sigma_{min}=50$. Then, behind $10^7$ draws consummated, estimates improvements are performed by the $Search$ step only.

The \mpmads curve shows an analogous behaviour to \dpmads. 
However, one can note the drop in the incumbent value from $10^5$ to $10^7$ draws. 
A more detailed analysis reveals that \mpmads encountered a nearly optimal basin of solutions and spent many draws exploring it. 
\mpmads is affected by this effect on almost all of its runs, 
    while \dpmads also visited this basin but did not waste as many draws in exploring it.
A possible explanation, recovered from the detailed logs, is that the low precision used by \dpmads actually helps it to avoid this basin, while \mpmads has a precision high enough to detect that it is interesting.

Figure~\ref{fig:VMEgrapheCV} may also be used to compare the quality of the solutions in terms to time elapsed.
After $10$ minutes of computing, 
    \dpmads has found a nearly optimal solution
    while \mpmads is still exploring near the $4000$ mark,
    and \robustmads is way below at $3000$.
It takes an entire day of computing for the three methods to reach a comparable solution.

\section{Discussion}
%===================================%

This work compares two generic strategies to optimize noisy problems. 
The monotonic \mpmads strategy avoid uncertainty as much as possible, leading to highly plausible data anytime in the optimization process (such as the localisation of the incumbent) with the drawback of an important computational effort. 
Following a different paradigm, the dynamic precision \dpmads algorithm reduces the computational effort per iteration but faces more uncertainty as a consequence. 
A noticeable point is the theoretical equivalence of these two strategies, as they share the same convergence analysis.

However, these two strategies behave differently in practical contexts. 
Both algorithms outperform strategies which do not exploit the adaptive precision. 
The \dpmads algorithm outperforms \mpmads on the test problems considered in this work, in the sense that it reaches an higher quality solution within a lower computational budget. 
Also, the two algorithms should not be considered for the same usage. 
Within a prescribed computational budget, the dynamic strategy tends to explore more solutions than the monotonic one. 
However, with the monotonic paradigm, the smaller set of evaluated points is well-known, in the sense that the estimated objective function value is more precise for all these points.

It should be noted that some improvements can be developed. 
Parameter values and implementation choices could be challenged. 
Notably, one could define some specific parameter values for given families of problems. 
In addition, precision growing to infinity leads to extremely long computational time in practical contexts, 
    therefore the monotonic strategy needs to be used with precaution. 
Meanwhile, the dynamic strategy struggles on problems with a flat objective function because it avoid as much as possible to improve the precision.
Also, usage of the precision index could be made more flexible: for example, its value could be modified at every evaluation rather than every iteration.

For future research, an important conceptual step would be the possibility to solve constrained problems, with constraints affected by an adaptive precision noise. Generalisation of the law followed by the noise, from centred normal to generic, could also be considered. The theory could also be enhanced with the addition of models to predict the behaviour of the objective.

Overall, the dynamic strategy could be chosen when one desires numerous solutions, while the monotonic strategy should be considered if one prefers fewer solutions but with an high precision on each. This comes from the fact that the dynamic strategy is designed to limit as much as possible the consumption of the computational budget per solution, while the monotonic allows more efforts per solution in order to rapidly identify the interesting ones.

%===================================%
\subsection*{Acknowledgements}
%===================================%

Thanks to the research and development team at EDF for sharing the asset management problem, 
and NSERC for its support through the CRD grant (\#RDCPJ 490744-15) with Hydro-Qu\'ebec and Rio~Tinto.

%===================================%
\bibliography{bibliography.bib}
%===================================%
\end{document}